\documentclass[reqno,10pt]{amsart}
\usepackage{amssymb}
\usepackage{amssymb, amsmath}
\usepackage{color}
\usepackage{verbatim}
\usepackage[margin=0.8in]{geometry}
\usepackage{comment}
\usepackage{latexsym}
\usepackage[makeroom]{cancel}

\def\ccite#1{~\cite{#1}}

\def\inte#1{
\displaystyle\mathop{#1\kern0pt}^\circ }

\def\Supp{\mathop{\rm Supp}\nolimits\ }

\let\pa=\partial
\let\al=\alpha

\let\d=\delta
\let\e=\varepsilon

\let\lam=\lambda

\let\s=\sigma
\let\f=\frac

\let\p=\psi

\let\D=\Delta

\let\wt=\widetilde
\let\wh=\widehat

\def\cB{{\mathcal B}}
\def\cC{{\mathcal C}}

\def\cE{{\mathcal E}}
\def\cF{{\mathcal F}}

\def\pa{\partial}

\def\dH{\dot{H}}

\def\h{{\rm h}}

\def\virgp{\raise 2pt\hbox{,}}
\def\cdotpv{\raise 2pt\hbox{;}}

\def\eqdefa{\buildrel\hbox{\footnotesize def}\over =}

\def\C{\mathop{\mathbb C\kern 0pt}\nolimits}
\def\DD{\mathop{\mathbb D\kern 0pt}\nolimits}
\def\EE{\mathop{{\mathbb E \kern 0pt}}\nolimits}
\def\K{\mathop{\mathbb K\kern 0pt}\nolimits}
\def\N{\mathop{\mathbb N\kern 0pt}\nolimits}
\def\Q{\mathop{\mathbb Q\kern 0pt}\nolimits}
\def\R{\mathop{\mathbb R\kern 0pt}\nolimits}
\def\SS{\mathop{\mathbb S\kern 0pt}\nolimits}
\def\ZZ{\mathop{\mathbb Z\kern 0pt}\nolimits}
\def\TT{\mathop{\mathbb T\kern 0pt}\nolimits}
\def\P{\mathop{\mathbb P\kern 0pt}\nolimits}

\newcommand{\la}{\lambda}

\newcommand{\Z}{{\ZZ}}

\def\dive{\mathop{\rm div}\nolimits}

\def\no{\noindent}
\def\na{\nabla}
\def\p{\partial}

\newcommand{\w}[1]{\langle {#1} \rangle}
\newcommand{\beq}{\begin{equation}}
\newcommand{\eeq}{\end{equation}}
\newcommand{\ben}{\begin{eqnarray}}
\newcommand{\een}{\end{eqnarray}}
\newcommand{\beno}{\begin{eqnarray*}}
\newcommand{\eeno}{\end{eqnarray*}}
\newcommand{\andf}{\quad\hbox{and}\quad}
\newcommand{\with}{\quad\hbox{with}\quad}
\newtheorem{defi}{Definition}[section]

\newtheorem{lem}{Lemma}[section]
\newtheorem{rmk}{Remark}[section]

\newcommand{\vv}[1]{\boldsymbol{#1}}

\def\v{{\rm v}}

\def\div{\text{div}\,}

\newtheorem*{Main Theorem}{Main Theorem}
\newtheorem{theorem}{Theorem}[section]
\newtheorem{lemma}[theorem]{Lemma}

\numberwithin{equation}{section}

\begin{document}
\title[Decay of 3D anisotropic Navier-Stokes equations]
{Enhanced dissipation for the third component of 3D anisotropic Navier-Stokes equations}

\author[L. XU]{Li Xu}
\address{School of Mathematical Sciences, Beihang University\\  100191 Beijing, China}
\email{xuliice@buaa.edu.cn}

\author[P. Zhang]{Ping Zhang}
\address{Academy of Mathematics $\&$ Systems Science
and  Hua Loo-Keng Key Laboratory of Mathematics, Chinese Academy of
Sciences, Beijing 100190, China, and School of Mathematical Sciences,
University of Chinese Academy of Sciences, Beijing 100049, China.}
\email{zp@amss.ac.cn}

\date{}

 \begin{abstract}
In this paper, we study the decay rates for the global small smooth  solutions to 3D anisotropic incompressible Navier-Stokes
equations. In particular, we prove that the horizontal components of the velocity field decay like the solutions of 2D classical Navier-Stokes equations.
While the third component of the velocity field decays as the solutions of 3D Navier-Stokes equations.  We remark that
such enhanced decay
rate for the third component is caused by the interplay between the divergence free condition of the velocity field and the horizontal
Laplacian in the anisotropic Navier-Stokes equations.
 \end{abstract}

\maketitle




\setcounter{equation}{0}
\section{Introduction}
We investigate the large time behavior of  the global small smooth solutions to the following 3D anisotropic Navier-Stokes equations
\beq\label{ANS}\left\{\begin{aligned}
&\p_t\vv v+\vv v\cdot\na\vv v-\Delta_{\h}\vv v+\na p=\vv 0,\quad(t,x)\in\R^+\times\R^3,\\
&\div\vv v=0,\\
&\vv v|_{t=0}=\vv v_0,
\end{aligned}\right.\eeq
where $\Delta_{\h}=\p_{x_1}^2+\p_{x_2}^2$  denotes the horizontal Laplacian operator,
 $\vv v=(v^1,v^2,v^3)$ and $p$ designate the velocity and the pressure of the fluid respectively.

Systems of this type appear in geophysical fluid dynamics (see for instance \cite{CDGG, Pedlovsky}).
In fact, meteorologists often modelize turbulent diffusion by putting
a viscosity of the form: $-\mu_\h\D_\h-\mu_3\pa_3^2$,
where $\mu_\h$ and $\mu_3$ are empirical constants,
and $\mu_3$ is usually much smaller than $\mu_\h$. For simplicity, we shall take $\mu_\h=1$ and $\mu_3=0$ in this context.

Just as the classical Navier-Stokes system where the horizontal Laplacian $\D_\h$ in \eqref{ANS}
is replaced by the full Laplician $\D=\sum_{j=1}^3\p_j^2,$ it is still big open question concerning
whether or not singularity will be developed in finite time for \eqref{ANS} with general initial data.
Considering system \eqref{ANS} has
 only horizontal dissipation, it is reasonable to use  the following anisotropic Sobolev space to study the well-posedness
 of the system \eqref{ANS}:

\begin{defi}\label{defanisob}
{\sl For any $(s,s')\in\R^2$, the anisotropic Sobolev space
$\dH^{s,s'}(\R^3)$ denotes the space of homogeneous tempered distribution
$a$ such that
$$\|a\|^2_{\dH^{s,s'}} \eqdefa \int_{\R^3} |\xi_{\rm
h}|^{2s}|\xi_3|^{2s'} |\wh a (\xi)|^2d\xi <\infty\with \xi_{\rm
h}=(\xi_1,\xi_2).$$
}\end{defi}

  Chemin et al. \cite{CDGG} and Iftimie \cite{Iftimie}
 proved that \eqref{ANS} is locally well-posed
with initial data in $L^2(\R^3)\cap \dH^{0,\f12+\varepsilon}(\R^3)$ for some $\varepsilon>0$,
and is globally well-posed if in addition
\begin{equation}\label{smallCDGG}
\|\vv v_0\|_{L^2}^\varepsilon
\|\vv v_0\|_{\dH^{0,\f12+\varepsilon}}^{1-\varepsilon}\leq c
\end{equation}
for some sufficiently small constant $c$.

Notice that just as the classical Navier-Stokes system,
the system \eqref{ANS} has the following scaling invariant property:
\begin{equation}\label{NSscaling}
\vv v_\lambda(t,x)\eqdefa\lambda \vv v(\lambda^2 t,\lambda x) \andf \vv v_{0,\lambda}(x)\eqdefa \lambda \vv v_0(\lambda x),
\end{equation}
which means  that if $\vv v$ is a solution of \eqref{ANS} with initial data $\vv v_0$ on $[0,T]$, $v_\la$ determined by \eqref{NSscaling}
is also a solution of \eqref{ANS} with initial data $\vv v_{0,\lam}$ on $[0,T/\la^2]$.

It is easy  to observe that
the smallness condition \eqref{smallCDGG}  in \cite{CDGG} is scaling invariant under the scaling transformation \eqref{NSscaling},
nevertheless,  the norm of the  space $H^{0,\f12+\varepsilon}$ is not. To investigate the well-posedness of  \eqref{ANS} with initial data in the critical spaces,
we  recall the following anisotropic dyadic operators from \cite{BCD}:
\begin{equation}\begin{split}\label{defparaproduct}
&\Delta_k^{\rm h}a\eqdefa\cF^{-1}(\varphi(2^{-k}|\xi_{\rm h}|)\widehat{a}),
 \quad \Delta_\ell^{\rm v}a \eqdefa\cF^{-1}(\varphi(2^{-\ell}|\xi_3|)\widehat{a}),\\
&S^{\rm h}_ka\eqdefa\cF^{-1}(\chi(2^{-k}|\xi_{\rm h}|)\widehat{a}),
\quad\ S^{\rm v}_\ell a \eqdefa\cF^{-1}(\chi(2^{-\ell}|\xi_3|)\widehat{a}),
\end{split}\end{equation}
where  $\xi_{\rm h}=(\xi_1,\xi_2),$ $\cF a$ or $\widehat{a}$  denotes the Fourier transform of $a$,
while $\cF^{-1} a$ designates the inverse Fourier transform of $a$,
$\chi(\tau)$ and $\varphi(\tau)$ are smooth functions such that
\begin{align*}
&\Supp \varphi \subset \Bigl\{\tau \in \R\,: \, \frac34 \leq
|\tau| \leq \frac83 \Bigr\}\quad\mbox{and}\quad \forall
 \tau>0\,,\ \sum_{j\in\Z}\varphi(2^{-j}\tau)=1;\\
& \Supp \chi \subset \Bigl\{\tau \in \R\,: \, |\tau|\leq
\frac43 \Bigr\}\quad\mbox{and}\quad \forall
 \tau\in\R\,,\ \chi(\tau)+ \sum_{j\geq 0}\varphi(2^{-j}\tau)=1.
\end{align*}

\begin{defi}\label{anibesov}
{\sl We define $\cB^{0,\f12}(\R^3)$ to be the set of  homogenous tempered distribution $a$ so that
$$
\|a\|_{\cB^{0,\f12}}\eqdefa \sum_{\ell\in\Z}2^{\f{\ell}2}
\|\D_\ell^\v a\|_{L^2(\R^3)}<\infty.
$$
}\end{defi}

Paicu \cite{Paicu} proved the local well-oosedness of the system \eqref{ANS} with initial data in $\cB^{0,\f12}$ and also
global well-posedness of \eqref{ANS}
provided that $\|\vv v_0\|_{\cB^{0,\f12}}$ is sufficiently small. Chemin and the second author \cite{CZ1} introduced a critical
 anisotropic Besov-Sobolev type space with negative index and proved the global well-posedness of \eqref{ANS} with small initial data
  in this space. In particular, this result implies the global well-posedness of \eqref{ANS} with highly oscillatory initial data in the horizontal variables. Those global well-posedness results in \cite{CZ1,Paicu} were improved in \cite{PZ1, Zhang10} that
the system has a unique global solution with two components of the initial data being small enough in the critical spaces.
Very recently, Liu, Paicu and the second author \cite{LZ4} proved the global well-posedness of \eqref{ANS} with uni-directional derivative
of the initial data being sufficiently small. More precisely, as long as $\||D_\h|^{-1}\p_3\vv v_0\|_{\cB^{0,\f12}}$ is small enough,
\eqref{ANS} has a unique global solution, which in particular improves the previous global well-posedness results in \cite{Paicu, PZ1, Zhang10}.
One may check \cite{LZ4} and the references therein for details concerning the well-posedness theory of \eqref{ANS}. The authors
of \cite{LZZ2} studied  the lower bound to the lifespan of the local smooth solutions to the system \eqref{ANS}.

Lately, Ji, Wu and Yang \cite{JWY} investigated the decay rate for the global smooth solution of \eqref{ANS}. More precisely,

\begin{theorem}[Theorem 1.1 of \cite{JWY}]\label{jwythm}
{\sl Let $\f34\leq\s<1.$ Assume that
\beno
\vv v_0\in H^4(\R^3),\quad \na\cdot \vv v_0=0,\quad \vv v_0, \p_3\vv v_0\in \dH^{-\s,0}(\R^3).
\eeno
Then there is $\e>0$ such that if
\beno
\|\vv v_0\|_{H^4}+\bigl\|(\vv v_0,\p_3\vv v_0)\bigr\|_{\dH^{-\s,0}}\leq \e,
\eeno
then \eqref{ANS} has a unique solution $\vv v$ satisfying
\begin{subequations} \label{jyw}
\begin{gather}
\label{jyw1}
\|\vv v\|_{{L}^\infty(\R^+; H^{4})}+ \|\vv v\|_{L^\infty(\R^+;\dH^{-\s,0})}\leq C\e,\\
\label{jyw2}
\|\vv v(t)\|_{{L}^2}+ \|\p_3\vv v(t)\|_{L^2}\leq C\e\w{t}^{-\f\s2},\\
\label{jyw3}
\|\na_\h\vv v(t)\|_{{L}^2}\leq C\e\w{t}^{-\f{\s+1}2}.
\end{gather}
\end{subequations}
}
\end{theorem}

We remark that although the linear part of the system \eqref{ANS} is 2D Stokes system, yet the Fourier splitting method
introduced by  Wiegner in\ccite {Wiegner}
(see also the works \ccite{HM06, Schonbek, Schonbek2}) to
study the decay-in-time
 estimates for 2D
   Navier-Stokes equation  can not be applied to investigate the decay rates for the global small solutions of \eqref{ANS}.
 Instead the authors of \cite{JWY} employed the integral formulation of the system \eqref{ANS} and used a continuous argument
 to prove Theorem \ref{jwythm}.

Our first observation in this paper is that under the critical smallness condition \eqref{initial ansatz}, we can propagate the regularity of
the solution to \eqref{ANS} in the
vertical variable.
We can also propagate the full high regularity of the solution to \eqref{ANS}. Precisely,
our first result  states as follows:

\begin{theorem}\label{global existence thm}
{\sl Let $s\geq 1$ and $\vv v_0\in \bigl(\cB^{0,\f12}\cap \dH^s\bigr)(\R^3)$ be a solenoidal vector field which satisfies
\beq\label{initial ansatz}
\|\vv v_0\|_{\cB^{0,\f12}}\leq c_0,
\eeq
for some $c_0>0$ sufficiently small. Then  the system \eqref{ANS} has a unique global solution $\vv v$  so that
\beno
\vv v\in C([0,\infty);(\cB^{0,\f12}\cap \dH^s)(\R^3))\andf  \na_{\h}\vv v\in L^2(\R_+; (\cB^{0,\f12}\cap \dH^s)(\R^3)),
\eeno
and
\begin{subequations} \label{S3eq0}
\begin{gather}
\label{S3eq2}
\|\vv v\|_{{L}^\infty(\R^+;\dH^{0,s})}^2+ \|\na_\h\vv v\|_{L^2(\R^+;\dH^{0,s})}^2\leq C\|\vv v_0\|_{\dH^{0,s}}^2,\\
\label{S3eq2a}
\|\vv v\|_{{L}^\infty(\R^+;\dH^{s})}^2+\|\na_\h\vv v\|_{L^2(\R^+;\dH^{s})}^2\leq C\|\vv v_0\|_{\dH^{s}}^2.
\end{gather}
\end{subequations}
Here and in what follows $C>0$ denotes a universal constant that may vary from line to line.
}
\end{theorem}

\medskip

Our second  result is concerned with the large time behavior of the global small solutions to the system \eqref{ANS}.

\begin{theorem}\label{decay theorem}
{\sl Let {$s_1>2$} and $s\in \bigl(\f{1+3s_1}{10(s_1-1)},1\bigr).$
 Let $\vv v_0\in \bigl(\dH^{0,s_1}\cap \dot{H}^{-s,0}\cap\dH^{-s,-\f{s}{2}-\f14}\bigr)(\R^3)$ with $\p_3\vv v_0\in\dot{H}^{-\f12,0}(\R^3).
$  We denote
\beq \label{S1eq1}
A_s\eqdefa\|\vv v_0\|_{L^2}^2+\|\vv v_0\|_{\dH^{-s,0}}^2 \andf B_s\eqdefa\|\vv v_0^3\|_{L^2}^2+ \|\vv v_0\|_{\dH^{-s,-\f{s}{2}-\f14}}^2+A_s^{\f32}.
\eeq
Then there exists $\varepsilon_0>0$ such that if besides \eqref{initial ansatz}, there hold in addition
\beq\label{initial ansatz a}
A_sB_s\leq \e_0 \andf
\mathcal{E}_0\eqdefa \|\p_3\vv v_0\|_{ \dot{H}^{-\f12,0}}^2+\|\vv v_0\|_{\dH^{0,s_1}}^2(A_sB_s)^{\f{s_1-1}{3s_1-2}}\leq \e_0.
\eeq
Then \eqref{ANS} has a unique global solution so that
\begin{itemize}
  \item[(1)] (propagation the regularity)
\begin{subequations} \label{S1eq2}
\begin{gather}
\label{S1eq2a} \|\vv v\|_{L^\infty(\R^+; \dot{H}^{-s,0})}^2+\|\na_{\h}\vv v\|_{L^2(\R^+; \dot{H}^{-s,0})}^2\leq C\|\vv v_0\|_{ \dot{H}^{-s,0}}^2,\\
\label{S1eq2b}
\|\vv v\|_{{L}^\infty(\R^+;\dH^{0,s_1})}^2+ \|\na_\h\vv v\|_{L^2(\R^+;\dH^{0,s_1})}^2\leq C\|\vv v_0\|_{\dH^{0,s_1}}^2,\\
\label{S1eq2c}
\|\p_3\vv v\|_{L^\infty(\R^+; \dot{H}^{-\f12,0})}^2+\|\p_3\vv v\|_{ L^2(\R^+;\dot{H}^{\f12,0})}^2\leq C\cE_0,
\end{gather}
\end{subequations}

\item[(2)] (decay estimates)
\begin{subequations} \label{S1eq3}
\begin{gather}
\label{S1eq3a}
\|\vv v(t)\|_{L^2}^2+t\|\na_\h\vv v(t)\|_{L^2}^2
\leq  C A_s\w{t}^{-s},\\
\label{S1eq3b}
\|\p_3\vv v(t)\|_{L^2}^2\leq e\bigl(\|\p_3\vv v_0\|_{L^2}^2+\cE_0
\bigr)\w{t}^{-\f12},
\end{gather}
\end{subequations}

\item[(3)] (enhanced dissipation for the third component)
\beq\label{S1eq4}
\|v^3(t)\|_{L^2}^2+t\|\na_\h v^3(t)\|_{L^2}^2\leq CB_s\w{t}^{-\f32s}t^{-\f14}
\eeq
\end{itemize}
}
\end{theorem}

\begin{rmk}
\begin{itemize}
  \item[(1)] The scaling invariant norm $\|\p_3\vv v\|_{L^2_t(\dH^{\f12,0})}$ is  crucial to propagate
the { negative } horizontal regularity of the solution to the system \eqref{ANS} (see \eqref{Estimate for v in H -s} below). That is the reason why we need the smallness condition on $\p_3\vv v_0$ in $\dH^{-\f12,0}(\R^3).$

\item[(2)] $s_1=4$ in Theorem \ref{decay theorem} corresponds to the assumption: $\vv v_0\in H^4(\R^3)$ in Theorem \ref{jwythm}.
Yet with $s_1=4,$ we have $s\in \left(\f{13}{30},1\right)$ in Theorem \ref{decay theorem}, which improves the assumption that $\s\in \left[\f34,1\right)$ in Theorem \ref{jwythm}. We do not know if the above result can be extended to any $s\in (0,1)$ as for the decay rates of
 classical 2D Navier-Stokes equations
(see for instance \cite{Wiegner}).

\item[(3)] We observe from \eqref{S1eq4} that the third component $v^3(t)$  of the $\vv v(t)$ decays faster than the horizontal components $\vv v^\h(t).$
This is purely due to the special structure of the third equation of \eqref{ANS}, precisely, \eqref{expression of v 3}, and the divergence
free condition of $\vv v.$ If we replace the norm
$\|\vv v_0\|_{\dH^{-s,-\f{s}{2}-\f14}}$ in Theorem \ref{decay theorem} by $\|\vv v_0\|_{\dH^{-s,-\f{s}{2}}},$ then  the solution $\vv v(t)$ decays according to
\beq\label{S1eq}
\|v^3(t)\|_{L^2}^2\leq CB_s\w{t}^{-\f32s} \andf \|\na_\h v^3(t)\|_{L^2}^2\leq CB_s\w{t}^{-\f32s}t^{-1},
\eeq
which is exactly as the decay rate for the global solutions of classical  3D Navier-Stokes system (see for instance \cite{Schonbek}).

\item[(4)] In general, the interplay between transport $\vv u\cdot\na$ and the diffusion of the following equation may cause the solution
$f$ decay more rapidly than the diffusion alone (see for instance \cite{CKRZ,CD20, Wei21})
\beno
\p_tf+\vv u\cdot\na f-\nu\D f=0.
\eeno
Here the enhanced decay rate of the third component $v^3$ is caused by interplay bewteen the divergence free condition $\p_1v^1+\p_2v^2+\p_3v^3=0,$
which transports the horizontal regularity of $\vv v^\h$ to the vertical regularity of $v^3,$ and the horizonal Laplacian in the $\vv v$ equations in \eqref{ANS}. That is the reason why we borrow the word "enhanced dissipation" from \cite{CKRZ,CD20, Wei21}.

\end{itemize}
\end{rmk}

\medskip

Let us end this section with some notations that will be used
throughout this paper.

\noindent{\bf Notations:}   Let $A, B$ be two operators, we denote
$[A;B]=AB-BA,$ the commutator between $A$ and $B$. For $a\lesssim b$,
we mean that there is a uniform constant $C,$ which may be different in each occurrence, such that $a\leq Cb$. We shall denote by~$(a|b)$
the $L^2(\R^3)$ inner product of $a$ and $b.$
 $\left(c_k\right)_{k\in\Z}$ (resp. $\left(c_k(t)\right)_{k\in\Z}$) designates a  generic elements on the unit
sphere of $\ell^2(\Z)$, i.e. $\sum_{k\in\Z}c^2_k=1$ (resp.  $\sum_{k\in\Z}c^2_k(t)=1$).
 Finally, we
denote $L^r_T(L^p_\h(L^q_\v))$ to be the space $L^r([0,T];
L^p(\R_{x_1}\times\R_{x_2}; L^q(\R_{x_3}))),$ and $\na_\h\eqdefa (\pa_{x_1},\pa_{x_2}),$ $\dive_\h=
\na_\h\cdot.$

\medskip

\setcounter{equation}{0}\label{Sect3}

\section{The propagation of regularities to Paicu's solution}

In this section, we shall prove the propagation of regularities for the solution obtained
by Paicu in \cite{Paicu} based on the estimate \eqref{S3eq1} below. Namely,
we are going to present the proof of Theorem \ref{global existence thm}.

\begin{proof}[Proof of Theorem \ref{global existence thm}] For simplicity, we just present the {\it a priori}
estimates for smooth enough solutions of \eqref{ANS}. Indeed
under the assumption of \eqref{initial ansatz}, we deduce from \cite{Paicu} that the system \eqref{ANS} has a unique global
solution $\vv v$ so that
\beq \label{S3eq1}
\|\vv v\|_{L^\infty(\R^+;\cB^{0,\f12})}+\|\na_\h\vv v\|_{L^2(\R^+;\cB^{0,\f12})}\leq C\|\vv v_0\|_{\cB^{0,\f12}}.
\eeq

With the estimate \eqref{S3eq1}, we can prove the propagation of regularity of $\vv v$ in the vertical variable.  We first get, by applying $\D_\ell^\v$ to the $\vv v$ equation of \eqref{ANS} and taking $L^2$
inner product of the resulting equation with $\D_\ell^\v \vv v,$ that
\beq\label{S3eq3}
\f12\f{d}{dt}\|\D_\ell^\v \vv v(t)\|_{L^2}^2+\|\na_\h\D_\ell^\v\vv v\|_{L^2}^2=
-\left(\D_\ell^\v(\vv v\cdot\na\vv v)\ |\ \D_\ell^\v\vv v\right).
\eeq

The estimate of term on the right-hand side (r.h.s.) of \eqref{S3eq3} relies on the following lemma:

\begin{lem}\label{S3lem1}
{\sl Let $s>0,$  one has
\beq \label{S3eq3a}
\begin{split}
\bigl|\left(\D_\ell^\v(\vv v\cdot\na\vv v)\ |\ \D_\ell^\v\vv v\right)\bigr|\lesssim c_\ell^2(t)2^{-2\ell s}
\Bigl(&\|\na_\h \vv v\|_{\cB^{0,\f12}}\|\vv v\|_{\dH^{0,s}}
\|\na_\h\vv v\|_{\dH^{0,s}}\\
&+\|\vv v\|_{\cB^{0,\f12}}^{\f12}\|\na_\h \vv v\|_{\cB^{0,\f12}}^{\f12}\|\vv v\|_{\dH^{0,s}}^{\f12}
\|\na_\h\vv v\|_{\dH^{0,s}}^{\f32}\Bigr).
\end{split}
\eeq
Here and in all that follows, we shall always denote $\bigl(c_\ell(t)\bigr)_{\ell\in\Z}$ to be a generic element of $\ell^2(\Z)$ so that
$\sum_{\ell\in\Z}c_\ell^2(t)=1.$
}
\end{lem}

We postpone the proof of this lemma till we finish the proof of Theorem \ref{global existence thm}.

By inserting the estimate \eqref{S3eq3a} into \eqref{S3eq3}, and then multiplying the inequality
by $2^{2\ell s},$ finally by summing up the resulting inequalities for $\ell$ in $\Z,$ we achieve
\beq \label{S3eq10}
\begin{split}
\f12\f{d}{dt}\|\vv v(t)\|_{\dH^{0,s}}^2+ \|\na_\h \vv v(t)\|_{\dH^{0,s}}^2\leq C\Bigl(&\|\na_\h \vv v\|_{\cB^{0,\f12}}\|\vv v\|_{\dH^{0,s}}
\|\na_\h\vv v\|_{\dH^{0,s}}\\
&+\|\vv v\|_{\cB^{0,\f12}}^{\f12}\|\na_\h \vv v\|_{\cB^{0,\f12}}^{\f12}\|\vv v\|_{\dH^{0,s}}^{\f12}
\|\na_\h\vv v\|_{\dH^{0,s}}^{\f32}\Bigr).
\end{split}
\eeq
Applying Young's inequality gives
\beno
\text{r.h.s.  of \eqref{S3eq10}}\leq  {C}\bigl(1+\|\vv v\|_{\cB^{0,\f12}}^{2}\bigr)\|\na_\h \vv v\|_{\cB^{0,\f12}}^2\|\vv v\|_{\dH^{0,s}}^2
+\f12\|\na_\h\vv v\|_{\dH^{0,s}}^2,
\eeno
so that there holds
\beq\label{S3eq10a}
\f{d}{dt}\|\vv v(t)\|_{\dH^{0,s}}^2+ \|\na_\h \vv v\|_{\dH^{0,s}}^2\leq {C}\bigl(1+\|\vv v\|_{\cB^{0,\f12}}^{2}\bigr)\|\na_\h \vv v\|_{\cB^{0,\f12}}^2\|\vv v\|_{\dH^{0,s}}^2.
 \eeq
 Applying Gronwwall's inequality yields for any $s>0$ that
 \begin{align*}
 \|\vv v\|_{L^\infty_t(\dH^{0,s})}^2+ \|\na_\h \vv v(t)\|_{L^2_t(\dH^{0,s)}}^2
 \leq \|\vv v_0\|_{\dH^{0,s}}^2\exp\Bigl({C}\bigl(1+\|\vv v\|_{L^\infty_t(\cB^{0,\f12})}^{2}\bigr)
 \|\na_\h \vv v\|_{L^2_t(\cB^{0,\f12})}^2\Bigr),
 \end{align*}
 which together with \eqref{S3eq1} leads to \eqref{S3eq2}.

To handle the  estimate of \eqref{S3eq2a} for case when $s\geq 1,$ we get, by applying $\D_k^\h$ to the $\vv v$ equation of \eqref{ANS} and then taking $L^2$-inner product
of the resulting equation with $\D_k^\h\vv v,$ that
\beq \label{S3eq12}
\f12\f{d}{dt}\|\D_k^\h\vv v(t)\|_{L^2}^2+ \|\na_\h\D_k^\h\vv v\|_{L^2}^2=-\left(\D_k^\h(\vv v\cdot\na\vv v)\ |\ \D_k^\h\vv v\right).
\eeq

The estimate of the term on the r.h.s. of \eqref{S3eq12} relies on the following lemma, the proof of which will be postponed at the end of this section.

\begin{lem}\label{S3lem2}
{\sl Let $s>-1,$  there hold
\beq \label{S3eq13}
\begin{split}
\bigl|\left(\D_k^\h(\vv v^\h\cdot\na_\h\vv v)\ |\ \D_k^\h\vv v\right)\bigr|\lesssim c_k^2(t)2^{-2k s}
\Bigl(&\|\na_\h \vv v\|_{\cB^{0,\f12}}\|\vv v\|_{\dH^{s,0}}
\|\na_\h\vv v\|_{\dH^{s,0}}\\
&+\|\vv v\|_{\cB^{0,\f12}}^{\f12}\|\na_\h \vv v\|_{\cB^{0,\f12}}^{\f12}\|\vv v\|_{\dH^{s,0}}^{\f12}
\|\na_\h\vv v\|_{\dH^{s,0}}^{\f32}\Bigr),
\end{split}
\eeq
and
\beq \label{S3eq14}
\begin{split}
\bigl|\left(\D_k^\h(v^3\pa_3\vv v)\ |\ \D_k^\h\vv v\right)\bigr|\lesssim &c_k^2(t)2^{-2ks}\Bigl\{\|\na_\h \vv v\|_{\cB^{0,\f12}}\|\vv v\|_{\dH^{s,0}}
\|\na_\h\vv v\|_{\dH^{s,0}}+\|\vv v\|_{\cB^{0,\f12}}^{\f12}\|\na_\h \vv v\|_{\cB^{0,\f12}}^{\f12}
\\
&\times\Bigl(\|\vv v\|_{\dH^{s,0}}^{\f12}
\|\na_\h\vv v\|_{\dH^{s,0}}^{\f32}+\|\na_\h\vv v\|_{\dH^{s,0}}\|\pa_3\vv v\|_{\dH^{s-1,0}}^{\f12}
\|\pa_3\vv v\|_{\dH^{s,0}}^{\f12}\Bigr)\Bigr\}.
\end{split}
\eeq
}
\end{lem}

By inserting the estimates \eqref{S3eq13} and \eqref{S3eq14} into \eqref{S3eq12}, and then multiplying the
resulting inequalities by $2^{2ks}$ and summing up $k$ over $\Z,$ we arrive at
\beq \label{S3eq15}
\begin{split}
\f12\f{d}{dt}\|\vv v(t)\|_{\dH^{s,0}}^2+ \|\na_\h\vv v(t)\|_{\dH^{s,0}}^2\leq C\Bigl(&\|\vv v\|_{\cB^{0,\f12}}^{\f12}\|\na_\h \vv v\|_{\cB^{0,\f12}}^{\f12}\bigl(\|\na_\h\vv v\|_{\dH^{s,0}}\|\pa_3\vv v\|_{\dH^{s-1,0}}^{\f12}
\|\pa_3\vv v\|_{\dH^{s,0}}^{\f12}\\
&+\|\vv v\|_{\dH^{s,0}}^{\f12}
\|\na_\h\vv v\|_{\dH^{s,0}}^{\f32}\bigr)+\|\na_\h \vv v\|_{\cB^{0,\f12}}\|\vv v\|_{\dH^{s,0}}
\|\na_\h\vv v\|_{\dH^{s,0}}\Bigr).
\end{split}
\eeq
However, notice that for $s\geq 1,$
\beno\begin{aligned}
\|\pa_3a\|_{\dH^{s-1,0}}^2=&\int_{\R^3}|\xi_\h|^{2(s-1)}|\xi_3|^2|\widehat{a}(\xi)|^2\,d\xi\\
\leq &\Bigl(\int_{\R^3}|\xi_\h|^{2s}|\widehat{a}(\xi)|^2\,d\xi\Bigr)^{\f{s-1}s}\Bigl(
\int_{\R^3}|\xi_3|^{2s}|\widehat{a}(\xi)|^2\,d\xi\Bigr)^{\f{1}s}=\|a\|_{\dH^{s,0}}^{\f{2(s-1)}s}\|a\|_{\dH^{0,s}}^{\f{2}s},
\end{aligned}\eeno
then we get, by applying Young's inequality, that
\beno\begin{aligned}
& C\|\vv v\|_{\cB^{0,\f12}}^{\f12}\|\na_\h \vv v\|_{\cB^{0,\f12}}^{\f12}\|\na_\h\vv v\|_{\dH^{s,0}}\|\pa_3\vv v\|_{\dH^{s-1,0}}^{\f12}
\|\pa_3\vv v\|_{\dH^{s,0}}^{\f12}\\
&\leq C\|\vv v\|_{\cB^{0,\f12}}^{\f12}\|\na_\h \vv v\|_{\cB^{0,\f12}}^{\f12}\|\na_\h\vv v\|_{\dH^{s,0}}
\|\vv v\|_{\dH^{s,0}}^{\f{s-1}{2s}}\|\vv v\|_{\dH^{0,s}}^{\f{1}{2s}}\|\na_\h\vv v\|_{\dH^{s,0}}^{\f{s-1}{2s}}\|\na_\h\vv v\|_{\dH^{0,s}}^{\f{1}{2s}}\\
&\leq {C}\|\vv v\|_{\cB^{0,\f12}}^{2}\|\na_\h \vv v\|_{\cB^{0,\f12}}^{2}\bigl(\|\vv v\|_{\dH^{s,0}}^{2}+\|\vv v\|_{\dH^{0,s}}^2\bigr)
+\f{1}6
\|\na_\h\vv v\|_{\dH^{s,0}}^2+\f12\|\na_\h\vv v\|_{\dH^{0,s}}^2,
\end{aligned}\eeno
\beno\begin{aligned}
&C\|\vv v\|_{\cB^{0,\f12}}^{\f12}\|\na_\h \vv v\|_{\cB^{0,\f12}}^{\f12}\|\vv v\|_{\dH^{s,0}}^{\f12}
\|\na_\h\vv v\|_{\dH^{s,0}}^{\f32}\leq {C}\|\vv v\|_{\cB^{0,\f12}}^{2}\|\na_\h \vv v\|_{\cB^{0,\f12}}^{2}\|\vv v\|_{\dH^{s,0}}^2
+\f{1}6
\|\na_\h\vv v\|_{\dH^{s,0}}^2,
\end{aligned}\eeno
and
\beno\begin{aligned}
&C\|\na_\h \vv v\|_{\cB^{0,\f12}}\|\vv v\|_{\dH^{s,0}}
\|\na_\h\vv v\|_{\dH^{s,0}}
\leq {C}\|\na_\h \vv v\|_{\cB^{0,\f12}}^{2}\|\vv v\|_{\dH^{s,0}}^2
+\f{1}6
\|\na_\h\vv v\|_{\dH^{s,0}}^2.
\end{aligned}\eeno
By substituting the above estimates into \eqref{S3eq15}, we achieve
\beno\begin{aligned}
\f{d}{dt}\|\vv v(t)\|_{\dH^{s,0}}^2+ \|\na_\h\vv v\|_{\dH^{s,0}}^2\leq & {C}\|\na_\h \vv v\|_{\cB^{0,\f12}}^{2}
\bigl(1+\|\vv v\|_{\cB^{0,\f12}}^{2}\bigr)\|\vv v\|_{\dH^{s,0}}^2\\
&+{C}\|\vv v\|_{\cB^{0,\f12}}^{2}\|\na_\h \vv v\|_{\cB^{0,\f12}}^{2}\|\vv v\|_{\dH^{0,s}}^2+\|\na_\h\vv v\|_{\dH^{0,s}}^2.
\end{aligned}\eeno
Applying Gronwall's inequality gives rise to
\beno\begin{aligned}
\|\vv v\|_{L^\infty_t(\dH^{s,0})}^2+ \|\na_\h\vv v\|_{L^2_t(\dH^{s,0})}^2\leq & \Bigl(\|\vv v_0\|_{\dH^{s,0}}^2
+{C}\|\vv v\|_{L^\infty_t(\cB^{0,\f12})}^{2}\|\na_\h \vv v\|_{L^2_t(\cB^{0,\f12})}^{2}\|\vv v\|_{L^\infty_t(\dH^{0,s})}^2\\
&+\|\na_\h\vv v\|_{L^2_t(\dH^{0,s})}^2\Bigr)\exp\Bigl(
{C}\|\na_\h \vv v\|_{L^2_t(\cB^{0,\f12})}^{2}
\bigl(1+\|\vv v\|_{L^\infty_t(\cB^{0,\f12})}^{2}\bigr)\Bigr),
\end{aligned}\eeno
which together with \eqref{S3eq1}, \eqref{initial ansatz} and \eqref{S3eq2} ensures
\beq
\label{S3eq11}
\|\vv v\|_{L^\infty_t(\dH^{s,0})}^2+ \|\na_\h\vv v\|_{L^2_t(\dH^{s,0})}^2\leq C\|\vv v_0\|_{\dH^{s}}^2.
\eeq

By combining \eqref{S3eq2} with \eqref{S3eq11}, we deduce \eqref{S3eq2a}.
This completes the proof of Theorem \ref{global existence thm}.
\end{proof}

Theorem \ref{global existence thm} has been proved provided that we provide the proof of Lemmas \ref{S3lem1} and
\ref{S3lem2}, which we present below.

\begin{proof}[Proof of Lemma \ref{S3lem1}]
Observing that
\beno
\left(\D_\ell^\v(\vv v\cdot\na\vv v)\ |\ \D_\ell^\v\vv v\right)=\left(\D_\ell^\v(\vv v^\h\cdot\na_\h\vv v)\ |\ \D_\ell^\v\vv v\right)
+\left(\D_\ell^\v(v^3\pa_3\vv v)\ |\ \D_\ell^\v\vv v\right),
\eeno
 we split the proof  of \eqref{S3eq3a} into the following two steps:

 \medskip

\no {\bf Step 1.} The estimate of $\left(\D_\ell^\v(\vv v^\h\cdot\na_\h\vv v)\ |\ \D_\ell^\v\vv v\right).$

\smallskip

By applying Bony's decomposition \eqref{bony} to $\vv v^\h\cdot\na_\h\vv v$ in the vertical variables, we find
\beno
\vv v^\h\cdot\na_\h\vv v=T^\v_{\vv v^\h}\cdot\na_\h\vv v+\wt{T}^\v_{\na_\h\vv v\cdot}\vv v^\h \with \wt{T}^\v_{\na_\h\vv v\cdot}\vv v^\h\eqdefa T^\v_{\na_\h\vv v\cdot}\vv v^\h+R^\v(\vv v^h,\na_h\vv v),
\eeno
Due to the support properties to the Fourier transform of the terms in $T^\v_{\vv v^\h}\cdot\na_\h\vv v,$ we
infer
\beno\begin{aligned}
\bigl|\left(\D_\ell^\v(T^\v_{\vv v^\h}\cdot\na_\h\vv v)\ |\ \D_\ell^\v\vv v\right)\bigr|
\lesssim & \sum_{|\ell'-\ell|\leq 4}\|S_{\ell'-1}^\v\vv v^\h\|_{L^4_\h(L^\infty_\v)}\|\D_{\ell'}^\v\na_\h\vv v\|_{L^2}\|\D_\ell^\v \vv v\|_{L^4_\h(L^2_\v)}\\
\lesssim & \sum_{|\ell'-\ell|\leq 4}\|\vv v^\h\|_{L^4_\h(L^\infty_\v)}\|\D_{\ell'}^\v\na_\h\vv v\|_{L^2}\|\D_\ell^\v \vv v\|_{L^2}^{\f12}
\|\D_\ell^\v\na_h\vv v\|_{L^2}^{\f12},
\end{aligned}\eeno
which together with the first inequality of Lemma \ref{interpolation lemma}, i.e.,
\beno
\|a\|_{L^4_\h(L^\infty_\v)}\lesssim \|a\|_{\cB^{0,\f12}}^{\f12}\|a\|_{\cB^{0,\f12}}^{\f12},
\eeno
ensures that
\beno\begin{aligned}
\bigl|\left(\D_\ell^\v(T^\v_{\vv v^\h}\na_\h\vv v)\ |\ \D_\ell^\v\vv v\right)\bigr|\lesssim & \Bigl(\sum_{|\ell'-\ell|\leq 4}c_{\ell'}(t)2^{-\ell's}\Bigr)c_\ell(t)2^{-\ell s}\|\vv v\|_{\cB^{0,\f12}}^{\f12}\|\na_\h \vv v\|_{\cB^{0,\f12}}^{\f12}\|\vv v\|_{\dH^{0,s}}^{\f12}
\|\na_\h\vv v\|_{\dH^{0,s}}^{\f32}\\
\lesssim &c_\ell^2(t)2^{-2\ell s}\|\vv v\|_{\cB^{0,\f12}}^{\f12}\|\na_\h \vv v\|_{\cB^{0,\f12}}^{\f12}\|\vv v\|_{\dH^{0,s}}^{\f12}
\|\na_\h\vv v\|_{\dH^{0,s}}^{\f32}.
\end{aligned}\eeno

Along the same line, due to $s>0,$ we get by using Lemma \ref{interpolation inequality} that
\beno\begin{aligned}
\bigl|\left(\D_\ell^\v(\wt{T}^\v_{\na_\h\vv v}{\vv v^\h})\ |\ \D_\ell^\v\vv v\right)\bigr|
\lesssim & \sum_{\ell'\geq \ell-N_0}\|S_{\ell'+2}^\v\na_\h\vv v^\h\|_{L^2_\h(L^\infty_\v)}\|\D_{\ell'}^\v\vv v\|_{L^4_\h(L^2_\v)}\|\D_\ell^\v \vv v\|_{L^4_\h(L^2_\v)}\\
\lesssim & \Bigl(\sum_{\ell'\geq \ell-N_0}c_{\ell'}(t)2^{-\ell's}\Bigr)c_\ell(t)2^{-\ell s}\|\na_\h \vv v\|_{\cB^{0,\f12}}\|\vv v\|_{\dH^{0,s}}
\|\na_\h\vv v\|_{\dH^{0,s}}\\
\lesssim &c_\ell^2(t)2^{-2\ell s}\|\na_\h \vv v\|_{\cB^{0,\f12}}\|\vv v\|_{\dH^{0,s}}
\|\na_\h\vv v\|_{\dH^{0,s}}.
\end{aligned}\eeno

By summarizing the above estimates, we arrive at
\beq \label{S3eq6}
\begin{split}
\bigl|\left(\D_\ell^\v(\vv v^\h\cdot\na_\h\vv v)\ |\ \D_\ell^\v\vv v\right)\bigr|\lesssim c_\ell^2(t)2^{-2\ell s}
\Bigl(&\|\na_\h \vv v\|_{\cB^{0,\f12}}\|\vv v\|_{\dH^{0,s}}
\|\na_\h\vv v\|_{\dH^{0,s}}\\
&+\|\vv v\|_{\cB^{0,\f12}}^{\f12}\|\na_\h \vv v\|_{\cB^{0,\f12}}^{\f12}\|\vv v\|_{\dH^{0,s}}^{\f12}
\|\na_\h\vv v\|_{\dH^{0,s}}^{\f32}\Bigr).
\end{split}
\eeq

\medskip

\no {\bf Step 2.} The estimate of $\left(\D_\ell^\v(v^3\pa_3\vv v)\ |\ \D_\ell^\v\vv v\right).$

We first get, by using a standard commutator's argument as that in \cite{CDGG,CZ1,Paicu}, that
\beq\label{S3eq7}
\begin{split}
\left(\D_\ell^\v(T_{v^3}^\v\pa_3\vv v)\ |\ \D_\ell^\v\vv v\right)=&\sum_{|\ell'-\ell|\leq 4}
\Bigl(\left([\D_\ell^\v; S_{\ell'-1}^\v v^3]\pa_3\D_{\ell'}^\v\vv v\ |\ \D_\ell^\v\vv v\right)\\
&\qquad+\left((S_{\ell'-1}^\v v^3-S_{\ell}^\v v^3)\pa_3\D_\ell^\v\D_{\ell'}^\v\vv v\ |\ \D_\ell^\v\vv v\right)\Bigr)
+\left(S_{\ell}^\v v^3\pa_3\D_\ell^\v\vv v\ |\ \D_\ell^\v\vv v\right).
\end{split}
\eeq
By applying   Lemmas \ref{commutator lemma} and \ref{L2} and using $\dive\vv v=0,$ we find
\beno\begin{aligned}
\sum_{|\ell'-\ell|\leq 4}&
\bigl|\left([\D_\ell^\v; S_{\ell'-1}^\v v^3]\pa_3\D_{\ell'}^\v\vv v\ |\ \D_\ell^\v\vv v\right)\bigr|\\
\lesssim & \sum_{|\ell'-\ell|\leq 4} 2^{-\ell}\|S_{\ell'-1}^\v\pa_3 v^3\|_{L^2_\h(L^\infty_\v)}\|\pa_3\D_{\ell'}^\v\vv v\|_{L^4_\h(L^2_\v)}
\|\D_{\ell}^\v\vv v\|_{L^4_\h(L^2_\v)}\\
\lesssim & \sum_{|\ell'-\ell|\leq 4} \|\dive_\h\vv v^\h\|_{L^2_\h(L^\infty_\v)}\|\D_{\ell'}^\v\vv v\|_{L^2}^{\f12}\|\na_\h\D_{\ell'}^\v\vv v\|_{L^2}^{\f12}\|\D_{\ell}^\v\vv v\|_{L^2}^{\f12}\|\na_\h\D_{\ell}^\v\vv v\|_{L^2}^{\f12}\\
\lesssim &c_\ell^2(t)2^{-2\ell s}\|\na_\h \vv v\|_{\cB^{0,\f12}}\|\vv v\|_{\dH^{0,s}}
\|\na_\h\vv v\|_{\dH^{0,s}}.
\end{aligned}\eeno
Similar estimate holds for the second term in \eqref{S3eq7}.

While we get, by using integration by parts
and $\dive v=0,$ that
\beno\begin{aligned}
\bigl|\left(S_{\ell}^\v v^3\pa_3\D_\ell^\v\vv v\ |\ \D_\ell^\v\vv v\right)\bigr|=&\f12\bigl|\int_{\R^3}
S_{\ell}^\v\pa_3 v^3|\D_\ell^\v\vv v|^2\,dx\bigr|\\
\leq&\f12\|\dive_\h\vv v^\h\|_{L^2_\h(L^\infty_\v)}\|\D_\ell^\v\vv v\|_{L^4_\h(L^2_\v)}^2\\
\lesssim &c_\ell^2(t)2^{-2\ell s}\|\na_\h \vv v\|_{\cB^{0,\f12}}\|\vv v\|_{\dH^{0,s}}
\|\na_\h\vv v\|_{\dH^{0,s}}.
\end{aligned}\eeno
This leads to
\beq \label{S3eq8}
\bigl|\left(\D_\ell^\v(T_{v^3}^\v\pa_3\vv v)\ |\ \D_\ell^\v\vv v\right)\bigr|\lesssim c_\ell^2(t)2^{-2\ell s}\|\na_\h \vv v\|_{\cB^{0,\f12}}\|\vv v\|_{\dH^{0,s}}
\|\na_\h\vv v\|_{\dH^{0,s}}. \eeq

On the other hand, again due to $\dive \vv v=0$ and $s>0,$ we deduce from Lemmas \ref{L2}  and \ref{interpolation inequality} that
\beno\begin{aligned}
\bigl|\left(\D_\ell^\v(\wt{T}^\v_{\pa_3\vv v}{v^3})\ |\ \D_\ell^\v\vv v\right)\bigr|
\lesssim & \sum_{\ell'\geq \ell-N_0}\|S_{\ell'+2}^\v\pa_3\vv v\|_{L^4_\h(L^\infty_\v)}\|\D_{\ell'}^\v v^3\|_{L^2}\|\D_\ell^\v \vv v\|_{L^4_\h(L^2_\v)}\\
\lesssim & \sum_{\ell'\geq \ell-N_0}\|S_{\ell'+2}^\v\vv v\|_{L^4_\h(L^\infty_\v)}\|\D_{\ell'}^\v\pa_3 v^3\|_{L^2}\|\D_\ell^\v \vv v\|_{L^4_\h(L^2_\v)}\\
\lesssim & \Bigl(\sum_{\ell'\geq \ell-N_0}c_{\ell'}(t)2^{-\ell's}\Bigr)c_\ell(t)2^{-\ell s}\|\vv v\|_{\cB^{0,\f12}}^{\f12}
\|\na_\h \vv v\|_{\cB^{0,\f12}}^{\f12}\|\vv v\|_{\dH^{0,s}}^{\f12}
\|\na_\h\vv v\|_{\dH^{0,s}}^{\f32}\\
\lesssim &c_\ell^2(t)2^{-2\ell s}\|\vv v\|_{\cB^{0,\f12}}^{\f12}
\|\na_\h \vv v\|_{\cB^{0,\f12}}^{\f12}\|\vv v\|_{\dH^{0,s}}^{\f12}
\|\na_\h\vv v\|_{\dH^{0,s}}^{\f32},
\end{aligned}\eeno
which together with \eqref{S3eq8} ensures that
\beq \label{S3eq9}
\begin{split}
\bigl|\left(\D_\ell^\v(v^3\pa_3\vv v)\ |\ \D_\ell^\v\vv v\right)\bigr|\lesssim c_\ell^2(t)2^{-2\ell s}
\Bigl(&\|\na_\h \vv v\|_{\cB^{0,\f12}}\|\vv v\|_{\dH^{0,s}}
\|\na_\h\vv v\|_{\dH^{0,s}}\\
&+\|\vv v\|_{\cB^{0,\f12}}^{\f12}\|\na_\h \vv v\|_{\cB^{0,\f12}}^{\f12}\|\vv v\|_{\dH^{0,s}}^{\f12}
\|\na_\h\vv v\|_{\dH^{0,s}}^{\f32}\Bigr).
\end{split}
\eeq

By combining \eqref{S3eq6} with \eqref{S3eq9}, we conclude the proof of \eqref{S3eq3a}.
\end{proof}

\begin{proof}[Proof of Lemma \ref{S3lem2}]
We divide the proof into the following two steps:

\no {\bf Step 1.} The estimate of $\left(\D_k^\h(\vv v^\h\cdot\na_\h\vv v)\ |\ \D_k^\h\vv v\right).$

By applying Bony's decomposition \eqref{bony} to $\vv v^\h\cdot\na_\h\vv v$ in the horizontal variables, we find
\beno
\vv v^\h\cdot\na_\h\vv v=T^\h_{\vv v^\h}\cdot\na_\h\vv v+{T}^\h_{\na_\h\vv v\cdot}\vv v^\h+R^\h(\vv v^\h,\na_\h\vv v).
\eeno
Due to the support properties to the Fourier transform of the terms in $T^\h_{\vv v^\h}\cdot\na_\h\vv v,$ we
infer by using Lemma \ref{interpolation lemma}  that
\beno\begin{aligned}
\bigl|\left(\D_k^\h(T^\h_{\vv v^\h}\cdot\na_\h\vv v)\ |\ \D_k^\h\vv v\right)\bigr|
\lesssim & \sum_{|k'-k|\leq 4}\|S_{k'-1}^\h\vv v^\h\|_{L^4_\h(L^\infty_\v)}\|\D_{k'}^\h\na_\h\vv v\|_{L^2}\|\D_k^\h \vv v\|_{L^4_\h(L^2_\v)}\\
\lesssim & \sum_{|k'-k|\leq 4}\|\vv v^\h\|_{L^4_\h(L^\infty_\v)}\|\D_{k'}^\h\na_\h\vv v\|_{L^2}\|\D_k^\h \na_\h\vv v\|_{L^2}^{\f12}
\|\D_k^\h \vv v\|_{L^2}^{\f12}\\
\lesssim &c_k^2(t)2^{-2k s}\|\vv v\|_{\cB^{0,\f12}}^{\f12}\|\na_\h \vv v\|_{\cB^{0,\f12}}^{\f12}\|\vv v\|_{\dH^{s,0}}^{\f12}
\|\na_\h\vv v\|_{\dH^{s,0}}^{\f32}.
\end{aligned}\eeno

Along the same line, we get
\beno\begin{aligned}
\bigl|\left(\D_k^\h({T}^\h_{\na_\h\vv v\cdot}{\vv v^\h})\ |\ \D_k^\h\vv v\right)\bigr|
\lesssim & \sum_{|k'-k|\leq 4}\|S_{k'-1}^\h\na_\h\vv v^\h\|_{L^2_\h(L^\infty_\v)}\|\D_{k'}^\h\vv v\|_{L^4_\h(L^2_\v)}\|\D_k^\h \vv v\|_{L^4_\h(L^2_\v)}\\
\lesssim & \Bigl(\sum_{|k'-k|\leq 4}c_{k'}(t)2^{-k's}\Bigr)c_k(t)2^{-k s}\|\na_\h \vv v\|_{\cB^{0,\f12}}\|\vv v\|_{\dH^{s,0}}
\|\na_\h\vv v\|_{\dH^{s,0}}\\
\lesssim &c_k^2(t)2^{-2k s}\|\na_\h \vv v\|_{\cB^{0,\f12}}\|\vv v\|_{\dH^{s,0}}
\|\na_\h\vv v\|_{\dH^{s,0}}.
\end{aligned}\eeno

While by applying Lemma \ref{L2} and using the fact that $s>-1,$ we find
\begin{align*}
\bigl|\left(\D_k^\h({R}^\h({\vv v^\h},{\na_\h\vv v}))\ |\ \D_k^\h\vv v\right)\bigr|
\lesssim & 2^k\sum_{k'\geq k-3}\|\D_{k'}^\h\vv v\|_{L^2}\|\wt{\D}_{k'}^\h\na_\h\vv v^\h\|_{L^2_\h(L^\infty_\v)}\|\D_k^\h \vv v\|_{L^2}\\
\lesssim & 2^k\Bigl(\sum_{k'\geq k-3}c_{k'}(t)2^{-k'(s+1)}\Bigr)c_k(t)2^{-k s}\|\na_\h \vv v\|_{\cB^{0,\f12}}\|\na_\h\vv v\|_{\dH^{s,0}}\|\vv v\|_{\dH^{s,0}}
\\
\lesssim &c_k^2(t)2^{-2k s}\|\na_\h \vv v\|_{\cB^{0,\f12}}\|\vv v\|_{\dH^{s,0}}
\|\na_\h\vv v\|_{\dH^{s,0}}.
\end{align*}

By summarizing the above estimates, we obtain \eqref{S3eq13}.

\medskip

\no {\bf Step 2.} The estimate of $\left(\D_k^\h(v^3\pa_3\vv v)\ |\ \D_k^\h\vv v\right).$

Once again we  get, by applying Bony's decomposition \eqref{bony} to $v^3\pa_3\vv v$ in the horizontal variables, that
\beq \label{S3eq13a}
v^3\pa_3\vv v=T^\h_{v^3}\pa_3\vv v+{T}^\h_{\pa_3\vv v}v^3+R^\h(v^3,\pa_3\vv v).
\eeq

By virtue of Lemma \ref{interpolation lemma}, it is easy to observe that
\beno\begin{aligned}
\bigl|\left(\D_k^\h(T^\h_{v^3}\pa_3\vv v)\ |\ \D_k^\h\vv v\right)\bigr|\lesssim
&\sum_{|k-k'|\leq 4}\|S_{k'-1}^\h v^3\|_{L^4_\h(L^\infty_\v)}\|\D_{k'}^\h\pa_3\vv v\|_{L^4_\h(L^2_\v)}\|\D_k^\h\vv v\|_{L^2}\\
\lesssim& \Bigl(\sum_{|k'-k|\leq 4}c_{k'}(t)2^{-k'(s-1)}\Bigr)c_k(t)2^{-k(s+1)}\\
&\qquad\qquad\times\|v^3\|_{L^4_\h(L^\infty_\v)}\|\pa_3\vv v\|_{\dH^{s-1,0}}^{\f12}
\|\pa_3\vv v\|_{\dH^{s,0}}^{\f12}\|\na_\h\vv v\|_{\dH^{s,0}}\\
\lesssim &c_k^2(t)2^{-2ks}\|v^3\|_{\cB^{0,\f12}}^{\f12}\|\na_\h v^3\|_{\cB^{0,\f12}}^{\f12}
\|\pa_3\vv v\|_{\dH^{s-1,0}}^{\f12}
\|\pa_3\vv v\|_{\dH^{s,0}}^{\f12}\|\na_\h\vv v\|_{\dH^{s,0}}.
\end{aligned}\eeno
While we get, by using integration by parts, that
\beno\begin{aligned}
\left(\D_k^\h({T}^\h_{\pa_3\vv v}{ v^3})\ |\ \D_k^\h\vv v\right)=-\left(\D_k^\h({T}^\h_{\vv v}{\pa_3 v^3})\ |\ \D_k^\h\vv v\right)
-\left(\D_k^\h({T}^\h_{\vv v}{v^3})\ |\ \D_k^\h\pa_3\vv v\right).
\end{aligned}\eeno
Using $\dive\vv v=0$ and Lemma \ref{interpolation lemma}, we find
\beno\begin{aligned}
\bigl|\left(\D_k^\h({T}^\h_{\vv v}{\pa_3v^3})\ |\ \D_k^\h\vv v\right)\bigr|
\lesssim &\sum_{|k'-k|\leq 4}\|S_{k'-1}^\h\vv v\|_{L^4_\h(L^\infty_\v)}\|\D_{k'}^\h\dive_\h\vv v^\h\|_{L^2}\|\D_k^\h\vv v\|_{L^4_\h(L^2_\v)}\\
\lesssim &\Bigl(\sum_{|k'-k|\leq 4} c_{k'}(t)2^{-k's}\Bigr)c_k(t)2^{-ks}\|\vv v\|_{L^4_\h(L^\infty_\v)}\|\na_\h \vv v^\h\|_{\dH^{s,0}}\|\vv v\|_{\dH^{s,0}}^{\f12}
\|\na_\h\vv v\|_{\dH^{s,0}}^{\f12}\\
\lesssim &c_k^2(t)2^{-2ks}\|\vv v\|_{\cB^{0,\f12}}^{\f12}\|\na_\h \vv v\|_{\cB^{0,\f12}}^{\f12}\|\vv v\|_{\dH^{s,0}}^{\f12}
\|\na_\h\vv v\|_{\dH^{s,0}}^{\f32}.
\end{aligned}\eeno
Similarly, one has
\beno\begin{aligned}
\bigl|\left(\D_k^\h({T}^\h_{\vv v}{v^3})\ |\ \D_k^\h\pa_3\vv v\right)\bigr|\lesssim &
 \sum_{|k'-k|\leq 4}\|S_{k'-1}^\h\vv v\|_{L^4_\h(L^\infty_\v)}\|\D_{k'}^\h v^3\|_{L^2}\|\D_k^\h \pa_3\vv v\|_{L^4_\h(L^2_\v)}\\
\lesssim &\Bigl(\sum_{|k'-k|\leq 4} c_{k'}(t)2^{-k'(s+1)}\Bigr)c_k(t)2^{-k(s-1)}\|\vv v\|_{L^4_\h(L^\infty_\v)}\\
&\qquad\qquad\qquad\qquad\times\|\na_\h v^3\|_{\dH^{s,0}}\|\pa_3\vv v\|_{\dH^{s-1,0}}^{\f12}\|\pa_3\vv v\|_{\dH^{s,0}}^{\f12}
\\
\lesssim &c_k^2(t)2^{-2ks}\|\vv v\|_{\cB^{0,\f12}}^{\f12}\|\na_\h \vv v\|_{\cB^{0,\f12}}^{\f12}\|\na_\h v^3\|_{\dH^{s,0}}\|\pa_3\vv v\|_{\dH^{s-1,0}}^{\f12}
\|\pa_3\vv v\|_{\dH^{s,0}}^{\f12}.
\end{aligned}\eeno

Along the same line, by using integration by parts, we find
\beno\begin{aligned}
\left(\D_k^\h({R}^\h({v^3},{\pa_3\vv v}))\ |\ \D_k^\h\vv v\right)=-\left(\D_k^\h(R^\h({\pa_3v^3},{\vv v}))\ |\ \D_k^\h\vv v\right)
-\left(\D_k^\h({R}^\h({v^3},{\vv v}))\ |\ \D_k^\h\pa_3\vv v\right).
\end{aligned}\eeno
Notice that $s>-1,$ by using $\dive\vv v=0$ and Lemma \ref{L2}, we obtain
\beno\begin{aligned}
\bigl|\left(\D_k^\h({R}^\h({\pa_3v^3},{\vv v}))\ |\ \D_k^\h\vv v\right)\bigr|
\lesssim &2^k\sum_{k'\geq k-3}\|\D_{k'}^\h\pa_3v^3\|_{L^2_\h(L^\infty_\v)}\|\wt{\D}_{k'}^\h\vv v\|_{L^2}\|\D_k^\h\vv v\|_{L^2}\\
\lesssim &2^k\Bigl(\sum_{k'\geq k-3} c_{k'}(t)2^{-k'(s+1)}\Bigr)c_k(t)2^{-ks}\|\na_\h\vv v^\h\|_{L^2_\h(L^\infty_\v)}
\|\na_\h\vv v^\h\|_{\dH^{s,0}}\|\vv v\|_{\dH^{s,0}}
\\
\lesssim &c_k^2(t)2^{-2ks}\|\na_\h \vv v\|_{\cB^{0,\f12}}\|\vv v\|_{\dH^{s,0}}
\|\na_\h\vv v\|_{\dH^{s,0}}.
\end{aligned}\eeno
Similarly, by virtue of Lemma \ref{interpolation lemma}, one has
\beno\begin{aligned}
\bigl|\left(\D_k^\h({R}^\h({v^3},{\vv v}))\ |\ \D_k^\h\pa_3\vv v\right)\bigr|\lesssim &
 \sum_{k'\geq k-3}\|\D_{k'}^\h\vv v\|_{L^4_\h(L^\infty_\v)}\|\wt{\D}_{k'}^\h v^3\|_{L^2}\|\D_k^\h \pa_3\vv v\|_{L^4_\h(L^2_\v)}\\
\lesssim &\Bigl(\sum_{k'\geq k-3} c_{k'}(t)2^{-k'(s+1)}\Bigr)c_k(t)2^{-k(s-1)}\|\vv v\|_{L^4_\h(L^\infty_\v)}\\
&\qquad\qquad\qquad\qquad\times\|\na_\h v^3\|_{\dH^{s,0}}\|\pa_3\vv v\|_{\dH^{s-1,0}}^{\f12}\|\pa_3\vv v\|_{\dH^{s,0}}^{\f12}
\\
\lesssim &c_k^2(t)2^{-2ks}\|\vv v\|_{\cB^{0,\f12}}^{\f12}\|\na_\h \vv v\|_{\cB^{0,\f12}}^{\f12}\|\na_\h v^3\|_{\dH^{s,0}}\|\pa_3\vv v\|_{\dH^{s-1,0}}^{\f12}
\|\pa_3\vv v\|_{\dH^{s,0}}^{\f12}.
\end{aligned}\eeno

By summarizing the above estimates, we obtain \eqref{S3eq14}. This finishes the proof of Lemma \ref{S3lem2}.
\end{proof}

\setcounter{equation}{0}
\section{Large time behavior of the global small solution to \eqref{ANS}}
In this section, we shall present the  proof of  Theorem \ref{decay theorem} concerning the large time behavior of the global small solution to \eqref{ANS},
especially the enhanced dissipation for the third component of the solution. In order to do so,
we need several auxiliary estimates which will be presented  in the following subsections.

\subsection{The estimate of $\|\na_{\h}\vv v(t)\|_{L^2}$}
The goal of this subsection is to present the $L^2$ estimate to
the horizontal derivatives  of $\vv v.$

\begin{lemma}\label{lemma for na h v}
{\sl Let $\vv v$ be a  smooth  enough solution of \eqref{ANS} on $[0,T].$ Then for $t\leq T,$ one has
\beq\label{Estimate for na h v 1}\begin{aligned}
&\f{d}{dt}\|\na_{\h}\vv v(t)\|_{L^2}^2+\|\D_{\h}\vv v\|_{L^2}^2\leq {C}\bigl(\|\na_{\h}\vv v\|_{\cB^{0,\f12}}^2
+\|\p_3\vv v\|_{\dH^{\f12,0}}^2\big)\|\na_{\h}\vv v\|_{L^2}^2.
\end{aligned}\eeq}
\end{lemma}
\begin{proof} By taking  $L^2$-inner product of the $\vv v$ equation of \eqref{ANS} with $-\D_{\h}\vv v$ and using integration by parts,
we find
\beq\label{E1}
\f12\f{d}{dt}\|\na_\h\vv v\|_{L^2}^2+\|\D_\h\vv v\|_{L^2}^2=\left(\vv v\cdot\na\vv v\,|\,\D_\h\vv v\right).
\eeq
To handle the term $\left(\vv v\cdot\na\vv v\,|\,\D_\h\vv v\right)$,  we get, by using integrating by parts
and $\dive \vv v=0,$ that
\beno\begin{aligned}
\left(\vv v\cdot\na\vv v\,|\,\D_\h\vv v\right)&=-\sum_{j=1}^2\bigl(\left(\p_j\vv v\cdot\na\vv v\,|\,\p_j\vv v\right)
+\left(\vv v\cdot\na\p_j\vv v\,|\,\p_j\vv v\right)\bigr)\\
&=-\sum_{j=1}^2\bigl(\left(\p_j\vv v^\h\cdot\na_\h\vv v\,|\,\p_j\vv v\right)
+\left(\p_jv^3\p_3\vv v\,|\,\p_j\vv v\right)\big).
\end{aligned}\eeno
Yet it follows  from \eqref{interpolation inequality} that
\beq\label{E2}
\begin{aligned}
\sum_{j=1}^2\bigl|\left(\p_j\vv v^\h\cdot\na_\h\vv v\,|\,\p_j\vv v\right)\bigr|\lesssim
&\|\na_\h\vv v^\h\|_{L^2_{\h}(L^\infty_\v)}\|\na_{\h}\vv v\|_{L^4_{\h}(L^2_{\v})}^2\\
\lesssim &\|\na_\h\vv v^h\|_{\cB^{0,\f12}}\|\na_{\h}\vv v\|_{L^2}\|\D_{\h}\vv v\|_{L^2}.
\end{aligned}\eeq

On the other hand, we observe that
\beq\label{E13}
\sum_{j=1}^2\bigl|\left(\p_jv^3\p_3\vv v\,|\,\p_j\vv v\right)\bigr|\lesssim\|\na_{\h}v^3\p_3\vv v\|_{L^2}\|\na_{\h}\vv v\|_{L^2}.
\eeq
Applying Bony's decomposition \eqref{bony} to $\na_{\h}v^3\p_3\vv v$ in the horizontal variables yields
\beno
\na_{\h}v^3\p_3\vv v=T^{\h}_{\na_{\h}v^3}\p_3\vv v+{T}^{\h}_{\p_3\vv v}{\na_{\h}v^3}+R^{\h}(\na_{\h}v^3,\p_3\vv v).
\eeno
For the term $T^{\h}_{\na_{\h}v^3}\p_3\vv v$, we have
\beno
\|\D_k^\h(T^{\h}_{\na_{\h}v^3}\p_3\vv v)\|_{L^2}\lesssim\sum_{|k-k'|\leq4}\|S_{k'-1}^{\h}\na_{\h}v^3\|_{L^\infty}\|\D_{k'}^{\h}\p_3\vv v\|_{L^2}
\eeno
Yet it follows from Lemma \ref{L2}, \eqref{interpolation inequality} and $\div\vv v=0$ that
\beno\begin{aligned}
\|S_{k'-1}^{\h}\na_{\h}v^3\|_{L^\infty}&\lesssim\sum_{l\leq k'-2}2^{2l}\|\D_{l}^{\h}v^3\|_{L_{\h}^2(L_{\v}^\infty)}
\lesssim\sum_{l\leq k'-2}2^{2l}\|\D_l^{\h}v^3\|_{L^2}^{\f12}\|\D_l^{\h}\p_3v^3\|_{L^2}^{\f12}\\
&\lesssim\sum_{l\leq k'-2}2^{2l}\|\D_l^{\h}v^3\|_{L^2}^{\f12}\|\D_l^{\h}(\na_{\h}\cdot\vv v^{\h})\|_{L^2}^{\f12}
\lesssim2^{\f{k'}{2}}\|\D_{\h}\vv v\|_{L^2},
\end{aligned}\eeno
from which, we infer
\beno\begin{aligned}
\|\D_k^{\h}(T^{\h}_{\na_{\h}v^3}\p_3\vv v)\|_{L^2}\lesssim &\|\D_{\h}\vv v\|_{L^2}\sum_{|k-k'|\leq4}2^{\f{k'}{2}}\|\D_{k'}^{\h}\p_3\vv v\|_{L^2}
\lesssim c_k(t)\|\D_{\h}\vv v\|_{L^2}\|\p_3\vv v\|_{\dH^{\f12,0}}.
\end{aligned}\eeno

Similarly, one has
\beno\begin{aligned}
\|\D_k^{\h}(T^{\h}_{\p_3\vv v}{\na_{\h}v^3})\|_{L^2}\lesssim & \sum_{|k-k'|\leq4}\|S_{k'-1}^{\h}\p_3\vv v\|_{L_{\v}^2(L_{\h}^\infty)}\|\D_{k'}^{\h}\na_{\h}v^3\|_{L_{\v}^\infty(L_{\h}^2)}\\
\lesssim&\|\p_3\vv v\|_{\dH^{\f12,0}}\cdot\sum_{|k'-k|\leq4} 2^{\f{k'}{2}}\|\D_{k'}^{\h}\na_{\h}v^3\|_{L^2}^{\f12}\|\D_{k'}^{\h}\na_{\h}\p_3v^3\|_{L^2}^{\f12}\\
\lesssim & c_k(t)\|\D_{\h}\vv v\|_{L^2}\|\p_3\vv v\|_{\dH^{\f12,0}}.
\end{aligned}\eeno

Finally we deduce from Lemma \ref{L2} and $\div \vv v=0$ that
\beno\begin{aligned}
\|\D_k^{\h}(R^{\h}(\na_{\h}v^3,\p_3\vv v))\|_{L^2}\lesssim &2^{k}\sum_{k'\geq k-3}\|\D_{k'}^{\h}\na_{\h}v^3\|_{L_{\v}^\infty(L_{\h}^2)}
\|\wt{\D}_{k'}^{\h}\p_3\vv v\|_{L^2}\\
\lesssim &2^{k}\sum_{k'\geq k-3}\|\D_{k'}^{\h}\na_{\h}v^3\|_{L^2}^{\f12}\|\D_{k'}^{\h}\na_{\h}\p_3v^3\|_{L^2}^{\f12}\cdot \|\wt{\D}_{k'}^{\h}\p_3\vv v\|_{L^2}\\
\lesssim &2^k\sum_{k'\geq k-3} c_{k'}(t)2^{-k'}\|\D_{\h}\vv v\|_{L^2}\|\p_3\vv v\|_{\dH^{\f12,0}}\\
\lesssim & c_k(t)\|\D_{\h}\vv v\|_{L^2}\|\p_3\vv v\|_{\dH^{\f12,0}}.
\end{aligned}\eeno

As a result, it comes out
\beno
\|\na_{\h}v^3\p_3\vv v\|_{L^2}\lesssim\|\p_3\vv v\|_{\dH^{\f12,0}}\|\D_{\h}\vv v\|_{L^2}.
\eeno
By inserting the above estimate into \eqref{E13}, we achieve
\beq\label{E14}
\sum_{j=1}^2\bigl|\left(\p_jv^3\p_3\vv v\,|\,\p_j\vv v\right)\bigr|\lesssim \|\p_3\vv v\|_{\dH^{\f12,0}}\|\D_{\h}\vv v\|_{L^2}\|\na_{\h}\vv v\|_{L^2}.
\eeq
Substituting \eqref{E2} and \eqref{E14} into \eqref{E1} gives rise to
\beno\begin{aligned}
&\f12\f{d}{dt}\|\na_{\h}\vv v\|_{L^2}^2+\|\D_{\h}\vv v\|_{L^2}^2\lesssim
\bigl(\|\na_{\h}\vv v\|_{\cB^{0,\f12}}
+\|\p_3\vv v\|_{\dH^{\f12,0}}\bigr)\|\na_{\h}\vv v\|_{L^2}\|\D_{\h}\vv v\|_{L^2},
\end{aligned}\eeno
which implies \eqref{Estimate for na h v 1}. This finishes the proof of Lemma \ref{lemma for na h v}.
\end{proof}

\subsection{The estimate of $\|\p_3\vv v(t)\|_{ \dot{H}^{-\f12,0}}$}
In order to close the estimate \eqref{Estimate for na h v 1}, we need to handle the estimate
of $\|\p_3\vv v\|_{L^2_t(\dot{H}^{\f12,0})},$ which is the purpose of this subsection.

\begin{lemma}\label{lemma for partial 3 v}
{Let $\vv v$ be a smooth enough solution of \eqref{ANS} on $[0,T].$ Then for any $t\leq T$, one has
\beq\label{Estimate for partial 3 v 1}\begin{aligned}
\f{d}{dt}\|\p_3\vv v(t)\|_{\dot{H}^{-\f12,0}}^2+\|\p_3\vv v\|_{\dot{H}^{\f12,0}}^2\leq {C}\Bigl(&\|\na_{\h}\vv v\|_{\cB^{0,\f12}}^2\|\p_3\vv v\|_{\dot{H}^{-\f12,0}}^2\\
&
+\|\na_{\h} v^3\|_{L^2}^{\f12}\|\dive_{\h}\vv v^\h\|_{L^2}^{\f12}\|\p_3^2\vv v\|_{L^2}\|\p_3\vv v\|_{\dH^{-\f12,0}}\Bigr).
\end{aligned}\eeq}
\end{lemma}
\begin{proof}
We first get, by applying the operator $\D_k^{\h}$ to  \eqref{ANS} and then taking $L^2$-inner product of the resulting
equations  with $-\p_3^2\D_k^{\h}\vv v,$ that
\beq\label{E20}
\f12\f{d}{dt}\|\D_k^{\h}\p_3\vv v\|_{L^2}^2+\|\na_{\h}\D_k^{\h}\p_3\vv v\|_{L^2}^2
=\left(\D_k^{\h}(\p_3\vv v\cdot\na\vv v)\,|\,\D_k^{\h}\p_3\vv v\right)+\left(\D_k^{\h}(\vv v\cdot\na\p_3\vv v)\,|\,\D_k^{\h}\p_3\vv v\right).
\eeq

\begin{itemize}
\item {\bf Estimate of $\left(\D_k^{\h}(\p_3\vv v\cdot\na\vv v)\,|\,\D_k^{\h}\p_3\vv v\right)$. }
\end{itemize}

 Due to $\div\vv v=0$, we write
\beq\label{E21}\begin{aligned}
\left(\D_k^{\h}(\p_3\vv v\cdot\na\vv v)\,|\,\D_k^{\h}\p_3\vv v\right)=\left(\D_k^{\h}(\p_3\vv v^{\h}\cdot\na_{\h}\vv v)\,|\,\D_k^{\h}\p_3\vv v\right)-\left(\D_k^{\h}((\div_{\h}\vv v^{\h})\p_3\vv v)\,|\,\D_k^{\h}\p_3\vv v\right).
\end{aligned}\eeq

 By applying Bony's decomposition \eqref{bony} to $\p_3\vv v^{\h}\cdot\na_{\h}\vv v$ in the horizontal variables,  we get
\beno
\p_3\vv v^{\h}\cdot\na_{\h}\vv v=T^\h_{\p_3\vv v^{\h}}\cdot\na_{\h}\vv v+{T}^{\h}_{\na_{\h}\vv v\cdot}{\p_3\vv v^{\h}}+R^\h(\p_3\vv v^{\h},\na_{\h}\vv v).
\eeno

For the term $T^{\h}_{\p_3\vv v^{\h}}\na_{\h}\vv v$, we have
\beno\begin{aligned}
\bigl|\left(\D_k^{\h}(T^\h_{\p_3\vv v^{\h}}\cdot\na_{\h}\vv v)\,|\,\D_k^{\h}\p_3\vv v\right)\bigr|
\lesssim &\sum_{|k-k'|\leq4}\|S_{k'-1}^{\h}\p_3\vv v^{\h}\|_{L^\infty_\h(L^2_\v)}\|\D_{k'}^{\h}\na_{\h}\vv v\|_{L^2_\h(L^\infty_\v)}\|\D_k^\h\pa_3v\|_{L^2}\\
\lesssim &2^{-k}\sum_{|k-k'|\leq4}2^{k'}\|S_{k'-1}^{\h}\p_3\vv v^{\h}\|_{L^2}\|\D_{k'}^{\h}\na_{\h}\vv v\|_{L^2_\h(L^\infty_\v)}\|\D_k^\h\na_\h\pa_3v\|_{L^2}\\
\lesssim & c_k^2(t)2^{k}\|\na_\h\vv v\|_{\cB^{0,\f12}}\|\pa_3\vv v\|_{\dH^{-\f12,0}}\|\na_\h\pa_3\vv v\|_{\dH^{-\f12,0}}.
\end{aligned}\eeno

While for term $T^\h_{\na_{\h}\vv v}{\p_3\vv v^{\h}},$ we get, by applying \eqref{interpolation inequality}, that
\beno\begin{aligned}
\bigl|\left(\D_k^{\h}(T^\h_{\na_{\h}\vv v\cdot}{\p_3\vv v^{\h}})\,|\,\D_k^{\h}\p_3\vv v\right)\big|
\lesssim &\sum_{|k-k'|\leq4}\|S_{k'-1}^{\h}\na_{\h}\vv v\|_{L^2_\h(L^\infty_\v)}\|\D_{k'}^{\h}\p_3\vv v^{\h}\|_{L^4_\h(L^2_\v)}
\|\D_{k}^{\h}\p_3\vv v^{\h}\|_{L^4_\h(L^2_\v)}\\
\lesssim & c_k^2(t)2^{k}\|\na_\h\vv v\|_{\cB^{0,\f12}}\|\pa_3\vv v\|_{\dH^{-\f12,0}}\|\na_\h\pa_3\vv v\|_{\dH^{-\f12,0}}.
\end{aligned}\eeno

Finally by applying Lemma \ref{L2}, we obtain
\beno\begin{aligned}
\bigl|\left(\D_k^{\h}(R^\h({\p_3\vv v^{\h}},\na_{\h}\vv v))\,|\,\D_k^{\h}\p_3\vv v\right)\bigr|
&\lesssim2^k\sum_{k'\geq k-3}\|\D_{k'}^{\h}\p_3\vv v^{\h}\|_{L^2}\|\wt{\D}_{k'}^{\h}\na_{\h}\vv v\|_{L^2_\h(L^\infty_\v)}\|\D_{k}^{\h}\p_3\vv v\|_{L^2}\\
&\lesssim2^{\f{3k}{2}}\Bigl(\sum_{k'\geq k-3}c_{k'}(t)2^{-\f{k'}2}\Bigr)c_k(t)\|\na_\h\pa_3\vv v\|_{\dH^{-\f12,0}}\|\na_\h\vv v\|_{\cB^{0,\f12}}\|\pa_3\vv v\|_{\dH^{-\f12,0}}\\
&\lesssim  c_k^2(t)2^{k}\|\na_\h\vv v\|_{\cB^{0,\f12}}\|\pa_3\vv v\|_{\dH^{-\f12,0}}\|\na_\h\pa_3\vv v\|_{\dH^{-\f12,0}}.
\end{aligned}\eeno
As a result, it comes out
\beq\label{E23}\begin{aligned}
\bigl|\left(\D_k^{\h}(\p_3\vv v^{\h}\cdot\na_{\h}\vv v)\,|\,\D_k^{\h}\p_3\vv v\right)\bigr|
\lesssim  c_k^2(t)2^{k}\|\na_\h\vv v\|_{\cB^{0,\f12}}\|\pa_3\vv v\|_{\dH^{-\f12,0}}\|\na_\h\pa_3\vv v\|_{\dH^{-\f12,0}}.
\end{aligned}\eeq
The estimate of the second term in \eqref{E21} is the same. Therefore, we obtain
\beq\label{E25}
\bigl|\left(\D_k^{\h}(\p_3\vv v\cdot\na\vv v)\,|\,\D_k^{\h}\p_3\vv v\right)\bigr|\lesssim  c_k^2(t)2^{k}\|\na_\h\vv v\|_{\cB^{0,\f12}}\|\pa_3\vv v\|_{\dH^{-\f12,0}}\|\na_\h\pa_3\vv v\|_{\dH^{-\f12,0}}.
\eeq

\medskip

\begin{itemize}
\item{\bf Estimate for term $\bigl|\left(\D_k^{\h}(\vv v\cdot\na\p_3\vv v)\,|\,\D_k^{\h}\p_3\vv v\right)\bigr|$. }
\end{itemize}

By applying  Bony's decomposition \eqref{bony} for $\vv v\cdot\na\p_3\vv v$ in the horizontal variable, we write
\beno
\vv v\cdot\na\p_3\vv v=T^\h_{\vv v}\cdot\na\p_3\vv v+T^\h_{\na\p_3\vv v\cdot}{\vv v}+R^\h({\vv v},\na\p_3\vv v).
\eeno
Let us first deal with the estimate of $(\D_k^{\h}(T^{\h}_{\vv v}\na\p_3\vv v)\,|\,\D_k^{\h}\p_3\vv v).$
Indeed we get, by using a standard commutator's argument, that
\beno\begin{aligned}
&(\D_k^{\h}(T^{\h}_{\vv v}\cdot\na\p_3\vv v)\,|\,\D_k^{\h}\p_3\vv v)=\sum_{|k-k'|\leq 4}\left(\D_k^{\h}(S_{k'-1}^{\h}\vv v\cdot\D_{k'}^{\h}\na\p_3\vv v)\,|\,\D_k^{\h}\p_3\vv v\right)
=I_k+II_k+III_k,
\end{aligned}\eeno
where
\beno\begin{aligned}
&I_k\eqdefa\sum_{|k-k'|\leq 4}\left(\D_k^{\h}([S_{k'-1}^{\h}\vv v-S_{k-1}^{\h}\vv v]\cdot\na\D_{k'}^{\h}\p_3\vv v)\,|\,\D_k^{\h}\p_3\vv v\right),\\
&II_k\eqdefa\sum_{|k-k'|\leq 4}\left([\D_k^{\h};S_{k-1}^{\h}\vv v]\cdot\na\D_{k'}^{\h}\p_3\vv v)\,|\,\D_k^{\h}\p_3\vv v\right),\\
&III_k\eqdefa\left(S_{k-1}^{\h}\vv v\cdot\na\D_k^{\h}\p_3\vv v\,|\,\D_k^{\h}\p_3\vv v\right).
\end{aligned}\eeno

Due to $\div\vv v=0$, we find
\beq\label{E26}
III_k=-\f12\left(\na\cdot S_{k-1}^{\h}\vv v\D_k^{\h}\p_3\vv v\,|\,\D_k^{\h}\p_3\vv v\right)=0.
\eeq

For $I_k$, we have
\beno\begin{aligned}
|I_k|&\lesssim\sum_{|k-k'|\leq 4}\sum_{|k-j|\leq 4}|(\D_k^{\h}(\D_j^{\h}\vv v\cdot\na\D_{k'}^{\h}\p_3\vv v)\,|\,\D_k^{\h}\p_3\vv v)|\\
&\lesssim\sum_{|k-k'|\leq 4}\sum_{|k-j|\leq 4}\Bigl(\|\D_j^{\h}\vv v^{\h}\|_{L^\infty}\|\na_{\h}\D_{k'}^{\h}\p_3\vv v\|_{L^2}
+\|\D_j^{\h}v^3\|_{L^\infty}\|\D_{k'}^{\h}\p_3^2\vv v\|_{L^2}\Bigr)\|\D_k^{\h}\p_3\vv v\|_{L^2}.
\end{aligned}\eeno
Yet it follows from Lemma \ref{L2} that
\beno\begin{aligned}
&\|\D_j^{\h}\vv v^{\h}\|_{L^\infty}\lesssim 2^j\|\D_j^{\h}\vv v^{\h}\|_{L^2_\h(L^\infty_\v)}
\lesssim\|\D_j^{\h}\na_{\h}\vv v\|_{L^2_\h(L^\infty_\v)}\lesssim\|\na_{\h}\vv v\|_{\cB^{0,\f12}},\\
&\|\D_j^{\h}v^3\|_{L^\infty}\lesssim 2^j\|\D_j^{\h}v^3\|_{L^2}^{\f12}\|\D_j^{\h}\p_3v^3\|_{L^2}^{\f12}
\lesssim2^{\f{j}{2}}\|\na_{\h}v^3\|_{L^2}^{\f12}\|\div_{\h}\vv v^\h\|_{L^2}^{\f12},
\end{aligned}\eeno
so that we obtain
\beq\label{E27}\begin{aligned}
|I_k|\lesssim&\|\na_{\h}\vv v\|_{\cB^{0,\f12}}\sum_{|k-k'|\leq 4}\|\D_{k'}^{\h}\na_\h\p_3\vv v\|_{L^2}\|\D_k^{\h}\p_3\vv v\|_{L^2}\\
&+\|\na_{\h} v^3\|_{L^2}^{\f12}\|\div_{\h}\vv v^\h\|_{L^2}^{\f12}\sum_{|k-k'|\leq 4}2^{\f{k'}{2}}\|\D_{k'}^{\h}\p_3^2\vv v\|_{L^2}\|\D_k^{\h}\p_3\vv v\|_{L^2}\\
\lesssim & c_k^2(t)2^{k}\bigl(\|\na_\h\vv v\|_{\cB^{0,\f12}}\|\na_\h\pa_3\vv v\|_{\dH^{-\f12,0}}+\|\na_{\h} v^3\|_{L^2}^{\f12}\|\div_{\h}\vv v^\h\|_{L^2}^{\f12}\|\p_3^2\vv v\|_{L^2}\bigr)\|\pa_3\vv v\|_{\dH^{-\f12,0}}.
\end{aligned}\eeq

For $II_k$, we get, by applying the commutator's estimate \eqref{commutator estimate}, that
\beno\begin{aligned}
&|II_k|\lesssim\sum_{|k-k'|\leq 4}2^{-k}\Bigl(\|\na_{\h}S_{k-1}^{\h}\vv v^{\h}\|_{L^\infty}\|\na_{\h}\D_{k'}^{\h}\p_3\vv v\|_{L^2}
+\|\na_{\h}S_{k-1}^{\h}v^3\|_{L^\infty}\|\D_{k'}^{\h}\p_3^2\vv v\|_{L^2}\Bigr)\|\D_k^{\h}\p_3\vv v\|_{L^2}.
\end{aligned}\eeno
Notice that
\beno\begin{aligned}
&\|\na_{\h}S_{k-1}^{\h}\vv v^{\h}\|_{L^\infty}\lesssim 2^k\|\na_{\h}S_{k-1}^{\h}\vv v^{\h}\|_{L^2_\h(L^\infty_{\v})}
\lesssim 2^k\|\na_{\h}\vv v^{\h}\|_{\cB^{0,\f12}},\\
&\|\na_{\h}S_{k-1}^{\h}v^3\|_{L^\infty}\lesssim 2^k\|\na_{\h}S_{k-1}^{\h}v^3\|_{L^2}^{\f12}\|\p_3\na_{\h}S_{k-1}^{\h}v^3\|_{L^2}^{\f12}
\lesssim2^{\f{3}{2}k}\|\na_{\h}v^3\|_{L^2}^{\f12}\|\div_{\h}\vv v^\h\|_{L^2}^{\f12},
\end{aligned}\eeno
from which, we deduce from a similar derivation of \eqref{E27}, that
\beq\label{E28}\begin{aligned}
|II_k| \lesssim  c_k^2(t)2^{k}\bigl(\|\na_\h\vv v\|_{\cB^{0,\f12}}\|\na_\h\pa_3\vv v\|_{\dH^{-\f12,0}}+\|\na_{\h}v^3\|_{L^2}^{\f12}\|\div_{\h}\vv v^\h\|_{L^2}^{\f12}\|\p_3^2\vv v\|_{L^2}\bigr)\|\pa_3\vv v\|_{\dH^{-\f12,0}}.
\end{aligned}\eeq

By summarizing the estimates \eqref{E26}, \eqref{E27} and \eqref{E28}, we obtain
\beq\label{E29}\begin{aligned}
\bigl|(\D_k^{\h}(T^{\h}_{\vv v}\cdot\na\p_3\vv v)\,|\,\D_k^{\h}\p_3\vv v)\bigr|\lesssim  c_k^2(t)2^{k}\bigl(&\|\na_\h\vv v\|_{\cB^{0,\f12}}\|\na_\h\pa_3\vv v\|_{\dH^{-\f12,0}}\\
&+\|\na_{\h}v^3\|_{L^2}^{\f12}\|\div_{\h}\vv v^\h\|_{L^2}^{\f12}\|\p_3^2\vv v\|_{L^2}\bigr)\|\pa_3\vv v\|_{\dH^{-\f12,0}}.
\end{aligned}\eeq

\smallskip

{For the term $(\D_k^{\h}(T^{\h}_{\na\p_3\vv v}{\vv v})\,|\,\D_k^{\h}\p_3\vv v)$,} we observe that
\beno\begin{aligned}
\bigl|\left(\D_k^{\h}({T}^{\h}_{\na\p_3\vv v}{\vv v})\,|\,\D_k^{\h}\p_3\vv v\right)\bigr|\lesssim
\sum_{|k-k'|\leq 4}\Bigl(&\|\na_{\h}S_{k'-1}^{\h}\p_3\vv v\|_{L^2}\|\D_{k'}^{\h}\vv v^{\h}\|_{L^\infty}\\
&+\|S_{k'-1}^{\h}\p^2_3\vv v\|_{L^\infty_{\h}(L^2_{\v})}\|\D_{k'}^{\h}v^3\|_{L^2_{\h}(L^\infty_{\v})}\Bigr)\|\D_k^{\h}\p_3\vv v\|_{L^2}.
\end{aligned}\eeno
It follows from Lemma \ref{L2} and $\div\vv v=0$, that
\beno\begin{aligned}
&\|\na_{\h}S_{k'-1}^{\h}\p_3\vv v\|_{L^2}\|\D_{k'}^{\h}\vv v^{\h}\|_{L^\infty}\lesssim c_{k'}(t)2^{\f{3}{2}k'}\|\na_\h \vv v^\h\|_{\cB^{0,\f12}}\|\p_3\vv v\|_{\dot{H}^{-\f12,0}},\\
&\|S_{k'-1}^{\h}\p^2_3\vv v\|_{L^\infty_{\h}(L^2_{\v})}\|\D_{k'}^{\h}v^3\|_{L^2_{\h}(L^\infty_{\v})}\lesssim
2^{k'}\|\p_3^2\vv v\|_{L^2}\|\D_{k'}^{\h}v^3\|_{L^2}^{\f12}\|\D_{k'}^{\h}\p_3v^3\|_{L^2}^{\f12}\\
&\qquad\qquad\qquad\qquad\qquad\qquad\qquad\ \lesssim c_{k'}(t)2^{\f{k'}{2}}\|\na_{\h}v^3\|_{L^2}^{\f12}\|\div_{\h}\vv v^\h\|_{L^2}^{\f12}\|\p_3^2\vv v\|_{L^2},
\end{aligned}\eeno
from which, we deduce that
\beq\label{E30}\begin{aligned}
\bigl|\left(\D_k^{\h}({T}^{\h}_{\na\p_3\vv v}{\vv v})\,|\,\D_k^{\h}\p_3\vv v\right)\bigr|\lesssim  c_k^2(t)2^{k}\bigl(&\|\na_\h\vv v\|_{\cB^{0,\f12}}\|\na_\h\pa_3\vv v\|_{\dH^{-\f12,0}}\\
&+\|\na_{\h}v^3\|_{L^2}^{\f12}\|\div_{\h}\vv v^\h\|_{L^2}^{\f12}\|\p_3^2\vv v\|_{L^2}\bigr)\|\pa_3\vv v\|_{\dH^{-\f12,0}}.
\end{aligned}\eeq

\smallskip

{ Finally we deal with the estimate of $(\D_k^{\h}(R^{\h}(\vv v\cdot,\na\p_3\vv v))\,|\,\D_k^{\h}\p_3\vv v)$.} Indeed by applying Lemma \ref{L2},
we find
\beno\begin{aligned}
&\bigl|\left(\D_k^{\h}(R^{\h}(\vv v,\na\p_3\vv v))\,|\,\D_k^{\h}\p_3\vv v\right)\bigr|\\
&\lesssim2^k\sum_{k'\geq k-3}\Bigl(\|\D_{k'}^{\h}\vv v^{\h}\|_{L^2_\h(L^\infty_{\v})}\|\na_{\h}\wt{\D}_{k'}^{\h}\p_3\vv v\|_{L^2}
+\|\D_{k'}^{\h}v^3\|_{L^2_{\h}(L^\infty_{\v})}\|\wt{\D}_{k'}^{\h}\p_3^2\vv v\|_{L^2}\Bigr)\|\D_k^{\h}\p_3\vv v\|_{L^2}\\
&\lesssim2^{\f{3k}2}\Bigl(\sum_{k'\geq k-3}c_{k'}(t)2^{-\f{k'}2}\Bigr)c_k(t)\Bigl(\|\na_\h\vv v\|_{\cB^{0,\f12}}\|\na_\h\pa_3\vv v\|_{\dH^{-\f12,0}}\\
&\qquad\qquad\qquad\qquad\qquad\qquad\qquad\ +\|\na_{\h} v^3\|_{L^2}^{\f12}\|\div_{\h}\vv v^\h\|_{L^2}^{\f12}\|\p_3^2\vv v\|_{L^2}\Bigr)\|\pa_3\vv v\|_{\dH^{-\f12,0}},
\end{aligned}\eeno
which implies
\beq\label{E31}\begin{aligned}
\bigl|\left(\D_k^{\h}(R^{\h}(\vv v\cdot,\na\p_3\vv v))\,|\,\D_k^{\h}\p_3\vv v\right)\bigr|
\lesssim  c_k^2(t)2^{k}\Bigl(&\|\na_\h\vv v\|_{\cB^{0,\f12}}\|\na_\h\pa_3\vv v\|_{\dH^{-\f12,0}}\\
&+\|\na_{\h}\vv v^3\|_{L^2}^{\f12}\|\div_{\h}\vv v^\h\|_{L^2}^{\f12}\|\p_3^2\vv v\|_{L^2}\Bigr)\|\pa_3\vv v\|_{\dH^{-\f12,0}}.
\end{aligned}\eeq

Thanks to \eqref{E29},  \eqref{E30} and  \eqref{E31},  we obtain
\beq\label{E32}\begin{aligned}
\bigl|\left(\D_k^\h(\vv v\cdot\na\p_3\vv v)\,|\,\D_k^{\h}\p_3\vv v\right)\bigr|
\lesssim  c_k^2(t)2^{k}\bigl(&\|\na_\h\vv v\|_{\cB^{0,\f12}}\|\na_\h\pa_3\vv v\|_{\dH^{-\f12,0}}\\
&+\|\na_{\h} v^3\|_{L^2}^{\f12}\|\div_{\h}\vv v^\h\|_{L^2}^{\f12}\|\p_3^2\vv v\|_{L^2}\bigr)\|\pa_3\vv v\|_{\dH^{-\f12,0}}.
\end{aligned}\eeq

\medskip

By inserting the estimates \eqref{E25} and \eqref{E32} into \eqref{E20} and  multiplying the resulting inequality by  $2^{-k}$ and summing  over $\Z$ with respect to $k$, we achieve
\beno\begin{aligned}
\f12\f{d}{dt}\|\p_3\vv v\|_{\dot{H}^{-\f12,0}}^2+\|\p_3\vv v\|_{\dot{H}^{\f12,0}}^2\leq C\bigl(&\|\na_\h\vv v\|_{\cB^{0,\f12}}\|\na_\h\pa_3\vv v\|_{\dH^{-\f12,0}}\\
&+\|\na_{\h}v^3\|_{L^2}^{\f12}\|\div_{\h}\vv v^\h\|_{L^2}^{\f12}\|\p_3^2\vv v\|_{L^2}\bigr)\|\pa_3\vv v\|_{\dH^{-\f12,0}},
\end{aligned}\eeno
which leads to \eqref{Estimate for partial 3 v 1}. This finishes the proof of  Lemma \ref{lemma for partial 3 v}.
\end{proof}

\subsection{The estimate of $\|\vv v(t)\|_{\dot{H}^{-s,0}}$}
In order to derive the decay in time estimate for the solutions of \eqref{ANS}, we need the negative derivative
estimate of $\vv v$ in the horizontal variables (see \eqref{s22} below).

\begin{lemma}\label{lemma for v in H -s}
{\sl Let $s\in (0,1)$ and $\vv v$ be a  smooth enough solution of \eqref{ANS} on $[0,T].$ Then for $t\leq T,$ we have
\beq\label{Estimate for v in H -s}\begin{aligned}
&\f{d}{dt}\|\vv v(t)\|_{\dot{H}^{-s,0}}^2+\|\na_{\h}\vv v\|_{\dot{H}^{-s,0}}^2\leq
{C}\bigl((1+\|\vv v\|_{\cB^{0,\f12}}^2)\|\na_{\h}\vv v\|_{\cB^{0,\f12}}^2
+\|\p_3\vv v\|_{\dH^{\f12,0}}^2\big)\|\vv v\|_{\dot{H}^{-s,0}}^2.
\end{aligned}\eeq}
\end{lemma}
\begin{proof} In view of \eqref{S3eq12} and \eqref{S3eq13}, it remains to handle the estimate of $\left(\D_k^\h(v^3\pa_3\vv v)\ |\ \D_k^\h\vv v\right).$ Indeed  due to $\div \vv v=0$ and $s>0,$ one has
\beno
\|S_{k-1}^\h v^3\|_{L^2_\h(L^\infty_\v)}\lesssim\sum_{j\leq k-2}\|\D_j^\h v^3\|_{L^2}^{\f12}\|\D_j^\h \p_3v^3\|_{L^2}^{\f12}
\lesssim  c_k(t)2^{ks}\|v^3\|_{\dH^{-s,0}}^{\f12}\|\div_\h\vv v^\h\|_{\dH^{-s,0}}^{\f12},
\eeno
so that we deduce
\beno\begin{aligned}
\bigl|\left(\D_k^\h(T^\h_{v^3}\pa_3\vv v)\ |\ \D_k^\h\vv v\right)\bigr|\lesssim
&\sum_{|k-k'|\leq 4}\|S_{k'-1}^\h v^3\|_{L^2_\h(L^\infty_\v)}\|\D_{k'}^\h\pa_3\vv v\|_{L^4_\h(L^2_\v)}\|\D_k^\h\vv v\|_{L^4_\h(L^2_\v)}\\
\lesssim& \Bigl(\sum_{|k'-k|\leq 4}c_{k'}(t)2^{k's}\Bigr)c_k(t)2^{ks}\\
&\qquad\qquad\times\|v^3\|_{\dH^{-s,0}}^{\f12}\|\div_\h\vv v^\h\|_{\dH^{-s,0}}^{\f12}\|\pa_3\vv v\|_{\dH^{\f12,0}}
\|\vv v\|_{\dH^{-s,0}}^{\f12}\|\na_\h \vv v\|_{\dH^{-s,0}}^{\f12}\\
\lesssim &c_k^2(t)2^{2ks}\|\pa_3\vv v\|_{\dH^{\f12,0}}
\|\vv v\|_{\dH^{-s,0}}\|\na_h\vv v\|_{\dH^{-s,0}}.
\end{aligned}\eeno
Similarly again due to $\dive\vv v=0,$ we find
\begin{align*}
\bigl|\left(\D_k^\h({T}^\h_{\pa_3\vv v}{v^3})\ |\ \D_k^\h\vv v\right)\bigr|
\lesssim &\sum_{|k'-k|\leq 4}\|S_{k'-1}^\h\pa_3\vv v\|_{L^4_\h(L^2_\v)}\|\D_{k'}^\h v^3\|_{L^2_\h(L^\infty_\v)}\|\D_k^\h\vv v\|_{L^4_\h(L^2_\v)}\\
\lesssim &\Bigl(\sum_{|k'-k|\leq 4} c_{k'}(t)2^{k's}\Bigr)c_k(t)2^{ks}\\
&\qquad\qquad\times\|\pa_3 \vv v\|_{L^4_\h(L^2_\v)}\|v^3\|_{\dH^{-s,0}}^{\f12}\|\div_\h\vv v^\h\|_{\dH^{-s,0}}^{\f12}\|\vv v\|_{\dH^{-s,0}}^{\f12}
\|\na_\h\vv v\|_{\dH^{-s,0}}^{\f12}\\
\lesssim &c_k^2(t)2^{2ks}\|\pa_3\vv v\|_{\dH^{\f12,0}}
\|\vv v\|_{\dH^{-s,0}}\|\na_\h\vv v\|_{\dH^{-s,0}}.
\end{align*}

While we get, by applying Lemma \ref{L2}
and $\dive\vv v=0,$ that
\beno
\|\D_{k'}^\h v^3\|_{L^2_\h(L^\infty_\v)}\lesssim \|\D_{k'}^\h v^3\|_{L^2}^{\f12}\|\D_{k'}^\h\p_3 v^3\|_{L^2}^{\f12}
\lesssim 2^{-\f{k'}2}\|\D_{k'}^\h\na_\h\vv v\|_{L^2}\lesssim c_{k'}(t)2^{k'\left(s-\f12\right)}\|\na_\h\vv v\|_{\dH^{-s,0}},
\eeno
so that by applying Lemma \ref{L2} once again, for $s<1,$ we infer
\begin{align*}
\bigl|\left(\D_k^\h({R}^\h(v^3,{\pa_3\vv v}))\ |\ \D_k^\h\vv v\right)\bigr|
\lesssim &2^k\sum_{k'\geq k-3}\|\D_{k'}^\h v^3\|_{L^2_\h(L^\infty_\v)}\|\wt{\D}_{k'}^\h\pa_3\vv v\|_{L^2}\|\D_k^\h\vv v\|_{L^2}\\
\lesssim &2^k\Bigl(\sum_{k'\geq k-3} c_{k'}(t)2^{k'(s-1)}\Bigr)c_k(t)2^{ks}\|\pa_3\vv v^\h\|_{\dH^{\f12,0}}\|\na_\h\vv v\|_{\dH^{-s,0}}
\|\vv v\|_{\dH^{-s,0}}
\\
\lesssim &c_k^2(t)2^{2ks}\|\pa_3 \vv v\|_{\dH^{\f12,0}}\|\vv v\|_{\dH^{-s,0}}
\|\na_\h\vv v\|_{\dH^{-s,0}}.
\end{align*}

Then in view of \eqref{S3eq13a}, we deduce that
\beq \label{S4eq1}
\bigl|\left(\D_k^\h(v^3\pa_3\vv v)\ |\ \D_k^\h\vv v\right)\bigr|\lesssim c_k^2(t)2^{-2ks}\|\pa_3 \vv v\|_{\dH^{\f12,0}}\|\vv v\|_{\dH^{-s,0}}
\|\na_\h\vv v\|_{\dH^{-s,0}}.
\eeq

By inserting the estimates \eqref{S3eq13} and \eqref{S4eq1} into \eqref{S3eq12}, and then multiplying the
resulting inequality by $2^{2ks}$ and summing up $k$ over $\Z,$ we arrive at
\beq \label{S4eq2}
\begin{split}
\f12\f{d}{dt}\|\vv v(t)\|_{\dH^{-s,0}}^2+\|\na_\h\vv v(t)\|_{\dH^{-s,0}}^2\lesssim C\Bigl(&\|\vv v\|_{\cB^{0,\f12}}^{\f12}\|\na_\h \vv v\|_{\cB^{0,\f12}}^{\f12}\|\vv v\|_{\dH^{-s,0}}^{\f12}
\|\na_\h\vv v\|_{\dH^{-s,0}}^{\f32}\\
&+\bigl(\|\na_\h \vv v\|_{\cB^{0,\f12}}+\|\pa_3 \vv v\|_{\dH^{\f12,0}}\bigr)\|\vv v\|_{\dH^{-s,0}}
\|\na_\h\vv v\|_{\dH^{-s,0}}\Bigr),
\end{split}
\eeq
which leads to \eqref{Estimate for v in H -s}. This finishes the proof of Lemma \ref{lemma for v in H -s}.
\end{proof}

\subsection{Proof of Theorem \ref{decay theorem}} We are now in a position to complete the proof of
Theorem \ref{decay theorem} which
  relies on the continuity argument and Lemmas \ref{lemma for na h v}, \ref{lemma for partial 3 v} and
 \ref{lemma for v in H -s}.

\begin{proof}[Proof of Theorem \ref{decay theorem}] Under the assumption of \eqref{initial ansatz}, we deduce
from Theorem \ref{global existence thm} that the system \eqref{ANS}
has  a unique global solution which satisfies
\beq
\label{S4eq2a}
\|\vv v\|_{{L}^\infty(\R^+;\dH^{0,s_1})}^2+ \|\na_\h\vv v\|_{L^2(\R^+;\dH^{0,s_1})}^2\leq C\|\vv v_0\|_{\dH^{0,s_1}}^2.
 \eeq
  Next, we shall only present the {\it a priori} estimate for the smooth enough solution of the system \eqref{ANS}.

In what follows, we divide the proof into the following steps:

\smallskip

\no{\bf Step 1. Ansatz for the continuity argument.}

\smallskip

We denote
\beq\label{S4eq4}
T^\star\eqdefa \sup\Bigl\{ \ T>0:\,\, \|\p_3\vv v\|_{L^\infty_T(\dot{H}^{-\f12,0})}^2+\|\p_3\vv v\|_{ L^2_T(\dot{H}^{\f12,0})}^2\leq
\e\ \Bigr\},
\eeq
where $\e$ is a small enough positive constant which will be determined later on.

 Then for $t\leq T^\star$ and $s\in (0,1),$ by using Gronwall's inequality, we deduce from \eqref{S3eq1} and \eqref{Estimate for v in H -s}
that
\beq\label{S4eq3}\begin{aligned}
\|\vv v\|_{L^\infty_t(\dot{H}^{-s,0})}^2+\|\na_{\h}\vv v\|_{L^2_t(\dot{H}^{-s,0})}^2\leq & \|\vv v_0\|_{\dot{H}^{-s,0}}^2
\exp\Bigl({C}\bigl(\|\na_{\h}\vv v\|_{L^2_t(\cB^{0,\f12})}^2
+\|\p_3\vv v\|_{L^2_t(\dH^{\f12,0})}^2\big)\Bigr)\\
\leq & \|\vv v_0\|_{\dot{H}^{-s,0}}^2
\exp\Bigl({C}\bigl(c_0^2
+\e\big)\Bigr)\leq e\|\vv v_0\|_{\dot{H}^{-s,0}}^2,
\end{aligned}\eeq
as long as $c_0$ in \eqref{initial ansatz} and $ \e$ in \eqref{S4eq4} are so small that $Cc_0^2\leq \f12$ and $C\e\leq \f12.$

\smallskip

\no{\bf Step 2. The decay estimate of $\|\vv v(t)\|_{L^2}$.}
\smallskip

Due to $\dive \vv v=0,$ we get, by taking $L^2$-inner product of the $\vv v$ equation of \eqref{ANS} with $\vv v,$ that
\beq \label{S4eq5}
\f12\f{d}{dt}\|\vv v(t)\|_{L^2}^2+\|\na_\h\vv v\|_{L^2}^2=0.
\eeq
While by applying H\"older's inequality in the frequency space and  using \eqref{S4eq3}, we find
\beno
\|\vv v(t)\|_{L^2}\leq\|\vv v(t)\|_{\dH^{-s,0}}^{\f{1}{1+s}}\|\na_{\h}\vv v(t)\|_{L^2}^{\f{s}{1+s}}
\leq e^{\f{1}{1+s}}\|\vv v_0\|_{\dH^{-s,0}}^{\f{1}{1+s}}\|\na_{\h}\vv v(t)\|_{L^2}^{\f{s}{1+s}},
\eeno
which implies
\beq\label{s22}
\|\na_{\h}\vv v(t)\|_{L^2}^2\geq \f{\|\vv v(t)\|_{L^2}^{\f{2(1+s)}{s}}}{(e\|\vv v_0\|_{\dH^{-s,0}})^{\f{2}{s}}}.
\eeq
Thanks to \eqref{S4eq5} and \eqref{s22}, we infer
\beno
\f{d}{dt}\|\vv v(t)\|_{L^2}^2+\f{2}{(e\|\vv v_0\|_{\dH^{-s,0}})^{\f{2}{s}}}\Bigl(\|\vv v(t)\|_{L^2}^2\Bigr)^{1+\f{1}{s}}\leq 0,
\eeno
Then for any $t\leq T^\star,$ we obtain
\beq \label{S4eq6}
\begin{split}
\|\vv v(t)\|_{L^2}^2\leq\Bigl(\f{1}{\|\vv v_0\|_{L^2}^{-\f{2}{s}}+ 2s^{-1}t(e\|\vv v_0\|_{\dH^{-s,0}})^{-\f{2}{s}}}\Bigr)^{s}
&\leq C A_s\w{t}^{-s}\with\\
& A_s\eqdefa \|\vv v_0\|_{L^2}^2+\|\vv v_0\|_{\dH^{-s,0}}^2.
 \end{split}
\eeq

\smallskip
\no{\bf Step 3. The decay estimate of $\|\na_{\h}\vv v(t)\|_{L^2}$.}
\smallskip

Motivated by the study of the decay-in-time estimate for the derivatives of the global solutions to classical
Navier-Stokes system (see \cite{HM06} for instance),
 for any $0\leq t_0<t\leq T^\star,$ we get, by multiplying $t-t_0$ to \eqref{Estimate for na h v 1}, that
\begin{align*}
&\f{d}{dt}\bigl[(t-t_0)\|\na_{\h}\vv v\|_{L^2}^2\bigr]+\bigl[(t-t_0)\|\D_{\h}\vv v\|_{L^2}^2\bigr]\leq \|\na_{\h}\vv v\|_{L^2}^2+ {C}
\bigl(\|\na_{\h}\vv v\|_{\cB^{0,\f12}}^2
+\|\p_3\vv v\|_{\dH^{\f12,0}}^2\big)(t-t_0)\|\na_{\h}\vv v\|_{L^2}^2.
\end{align*}
By applying Gronwall's inequality and using \eqref{initial ansatz} and \eqref{S4eq4} for sufficiently small $c_0$ and $\e,$
we find
\beq \label{S4eq7}
\begin{aligned}
&(t-t_0)\|\na_{\h}\vv v(t)\|_{L^2}^2{+\int_{t_0}^t(t'-t_0)\|\D_{\h}\vv v(t')\|_{L^2}^2\,dt'}\\
&\leq  \int_{t_0}^t\|\na_{\h}\vv v(t')\|_{L^2}^2\,d t'\exp\Bigl( {C}
\bigl(\|\na_{\h}\vv v\|_{L^2_t(\cB^{0,\f12})}^2
+\|\p_3\vv v\|_{L^2_t(\dH^{\f12,0})}^2\big)\Bigr)\\
&\leq e  \int_{t_0}^t\|\na_{\h}\vv v(t')\|_{L^2}^2\,d t'.
\end{aligned}\eeq
On the other hand, we get, by integrating \eqref{S4eq5} over $[\tau,t]$ for any $0\leq\tau<t,$ that
\beno
\int_\tau^t\|\na _\h\vv v(t')\|_{L^2}^2\,dt'\leq \f12\|\vv v(\tau)\|_{L^2}^2,
\eeno
from which, \eqref{S4eq6} and \eqref{S4eq7}, we deduce that
\beq \label{S4eq8}\begin{aligned}
&\|\na_{\h}\vv v(t)\|_{L^2}^2+{\f{2}{t}\int_{\f{t}{2}}^t(t'-\f{t}{2})\|\D_{\h}\vv v\|_{L^2}^2dt'}\\
&\quad\leq \f{2e}t\int_{\f{t}2}^t\|\na_{\h}\vv v(t')\|_{L^2}^2\,d t'\leq \f{e}t\|\vv v(t/2)\|_{L^2}^2\leq
C_sA_s\w{t}^{-s}t^{-1}\quad\forall\ t\leq T^\star.
\end{aligned}\eeq
This together with \eqref{S4eq6} ensures that \eqref{S1eq3a} holds for $t\leq T^\star.$

\smallskip
\no
{\bf Step 4. The decay estimate of $\|v^3(t)\|_{L^2}.$}

\smallskip

By taking space divergence to the $\vv v$ equation of \eqref{ANS}, we write
\beno
-\D p=\na\cdot(\vv v\cdot\na\vv v)=\sum_{j,k=1}^3\p_j\p_k(v^jv^k),
\eeno
which yields
\beq\label{expression of p}
\na p=\na(-\D)^{-1}\sum_{j,k=1}^3\p_j\p_k(v^jv^k).
\eeq

While by applying Fourier transform to the $v^3$ equation of \eqref{ANS}, we find
\beno
\p_t\widehat{v}^3+|\xi_{\h}|^2\widehat{v}^3=-\mathcal{F}\left(\vv v\cdot \na v^3+\p_3p\right).
\eeno
Using $\div\vv v=0$ and \eqref{expression of p}, we infer
\beq\label{expression of v 3}\begin{aligned}
\widehat{v}^3(t,\xi)&=e^{-t|\xi_{\h}|^2}\widehat{v}_0^3(\xi)+\int_0^te^{-(t-\tau)|\xi_{\h}|^2}F_1(\tau,\xi)d\tau
+\int_0^te^{-(t-\tau)|\xi_{\h}|^2}F_2(\tau,\xi)d\tau\\
&
\eqdefa \widehat{v_L^3}(t,\xi)+\widehat{v_{N1}^3}(t,\xi)+\widehat{v_{N2}^3}(t,\xi) \with\\
F_1(t,\xi)&\eqdefa -i|\xi|^{-2}|\xi_{\h}|^2\sum_{j=1}^3\xi_j\widehat{v^j v^3}(t,\xi) \andf F_2(t,\xi)\eqdefa i\xi_3|\xi|^{-2}\sum_{j=1}^3\sum_{k=1}^2\xi_j\xi_k\widehat{v^jv^k}(t,\xi).
\end{aligned}\eeq

By taking $L^2$-inner product of \eqref{expression of v 3} with $\widehat{v}^3(t,\xi)$, we obtain
\beno
\|\widehat{v}^3(t,\cdot)\|_{L^2}^2\leq\|\widehat{v^3_L}(t,\cdot)\|_{L^2}\|\widehat{v}^3(t,\cdot)\|_{L^2}
+\left(\widehat{v^3_{N1}}(t,\xi)+\widehat{v^3_{N2}}(t,\xi)\,|\,\widehat{v}^3(t,\xi)\right),
\eeno
from which, we infer
\beq\label{A1}
\|\widehat{v}^3(t,\cdot)\|_{L^2}^2\leq\|\widehat{v^3_L}(t,\cdot)\|_{L^2}^2
+2\bigl|\bigl(\widehat{v^3_{N1}}(t,\xi)+\widehat{v^3_{N2}}(t,\xi)\,|\,\widehat{v}^3(t,\xi)\bigr)\bigr|.
\eeq

\smallskip

\begin{itemize}
\item[(1)] { Decay estimate of $\|v^3_L(t)\|_{L^2}$.}
\end{itemize}
\smallskip

 By virtue of the definition of $\widehat{v^3_L}(t,\xi)$ in \eqref{expression of v 3}, we write
\beq\label{s25}\begin{aligned}
\|\widehat{v^3_L}(t)\|_{L^2}^2&=\int_{\R^3}e^{-2 t|\xi_{\h}|^2}|\widehat{v}_0^3(\xi)|^2\,d\xi\\
&=\int_{|\xi_{\h}|\leq|\xi_3|}e^{-2t|\xi_{\h}|^2}|\widehat{v}_0^3(\xi)|^2\,d\xi
+\int_{|\xi_{\h}|\geq|\xi_3|}e^{-2 t|\xi_{\h}|^2}|\widehat{v}_0^3(\xi)|^2\,d\xi\\
&\eqdefa A_1+A_2.
\end{aligned}\eeq
Due to $s\in (0,1),$ we  get, by using $\div\vv v_0=0$, that
\beno\begin{aligned}
A_1&=\int_{|\xi_{\h}|\leq|\xi_3|}e^{-2 t|\xi_{\h}|^2}\f{1}{|\xi_3|^2}\cdot|i\xi_3\widehat{v}_0^3(\xi)|^2d\xi\\
&\leq\int_{\R^2}\int_{|\xi_3|\geq|\xi_{\h}|}\f{|\widehat{\vv v}_0^{\h}(\xi_{\h},\xi_3)|^2}{|\xi_3|^2}d\xi_3
\cdot e^{-2t|\xi_{\h}|^2}|\xi_{\h}|^2d\xi_{\h}\\
&\lesssim\int_{\R^2}e^{-2 t|\xi_{\h}|^2}|\xi_{\h}|^{3s+\f12}\cdot|\xi_{\h}|^{-2s}\int_{\R}|\widehat{\vv v_0^{\h}}(\xi_{\h},\xi_3)|^2|\xi_3|^{-\left(s+\f12\right)}\,d\xi_3\,d\xi_{\h}\\
&\lesssim t^{-\left(\f{3s}{2}+\f14\right)}\int_{\R^3}|\widehat{\vv v_0^{\h}}(\xi_{\h},\xi_3)|^2|\xi_{\h}|^{-2s}|\xi_3|^{-\left(s+\f12\right)}d\xi,
\end{aligned}\eeno
which implies
\beq\label{s23}
A_1\lesssim t^{-\left(\f{3s}{2}+\f14\right)}\|\vv v_0^{\h}\|_{\dH^{-s,-\f{s}{2}-\f14}}^2.
\eeq

While we observe that
\beno\begin{aligned}
A_2&=\int_{|\xi_{\h}|\geq|\xi_3|}e^{-2 t|\xi_{\h}|^2}|\xi_{\h}|^{2s}|\xi_3|^{s+\f12}\cdot|\xi_{\h}|^{-2s}|\xi_3|^{-\left(s+\f12\right)}|\widehat{v}_0^3(\xi)|^2\,d\xi\\
&\lesssim\int_{\R^3}e^{-2 t|\xi_{\h}|^2}|\xi_{\h}|^{3s+\f12}\cdot|\xi_{\h}|^{-2s}|\xi_3|^{-\left(s+\f12\right)}|\widehat{v}_0^3(\xi)|^2\,d\xi\\
&\lesssim t^{-\left(\f{3s}{2}+\f14\right)}\int_{\R^3}|\widehat{v}_0^3(\xi_{\h},\xi_3)|^2|\xi_{\h}|^{-2s}|\xi_3|^{-\left(s+\f12\right)}d\xi,
\end{aligned}\eeno
which gives rise to
\beq\label{s24}
A_2\lesssim t^{-\left(\f{3s}{2}+\f14\right)}\|v_0^3\|_{\dH^{-s,-\f{s}{2}-\f14}}^2.
\eeq

By inserting the estimates \eqref{s23} and \eqref{s24} into  \eqref{s25}, we obtain
\beno
\|v^3_L(t)\|_{L^2}\lesssim t^{-\left(\f{3s}{4}+\f18\right)}\|\vv v_0\|_{\dH^{-s,-\f{s}{2}-\f14}},
\eeno
which together with the fact that $\|v^3_L(t)\|_{L^2}\leq \|v_0^3\|_{L^2}$ ensures that
\beq\label{decay estimate for the linear part of v 3}
\|v^3_L(t)\|_{L^2}\leq C \bigl(\|v_0^3\|_{L^2}+ \|\vv v_0\|_{\dH^{-s,-\f{s}{2}-\f14}}\bigr)\w{t}^{-\left(\f{3s}{4}+\f18\right)}.
\eeq

\smallskip

\begin{itemize}
\item[(2)]
{ Decay estimate for  term involving $v^3_{N1}+v^3_{N2}$.}
\end{itemize}

\smallskip

By virtue  of the definition of $F_2(t,\xi)$ in \eqref{expression of v 3}, we get, by using $\div\vv v=0$, that
\beno\begin{aligned}
\bigl(\widehat{v^3_{N2}}(t,\xi)\,|\,\widehat{v}^3(t,\xi)\bigr)
&=
-\Bigl(\int_0^te^{-(t-\tau)|\xi_{\h}|^2}|\xi|^{-2}\sum_{j=1}^3\sum_{k=1}^2\xi_j\xi_k\widehat{v^jv^k}(\tau,\xi)\,d\tau\,|\,i\xi_3\widehat{v}^3(t,\xi)\Bigr)\\
&=\Bigl(\int_0^te^{- (t-\tau)|\xi_{\h}|^2}i\xi_{\h}|\xi|^{-2}\sum_{j=1}^3\sum_{k=1}^2\xi_j\xi_k\widehat{v^jv^k}(\tau,\xi)\,d\tau\,|\,\widehat{\vv v^{\h}}(t,\xi)\Bigr),
\end{aligned}\eeno
from which and the definition of $F_1(t,\xi)$ in \eqref{expression of v 3}, we infer
\beq\label{A2}
\bigl|\bigl(\widehat{v^3_{N1}}(t,\xi)+\widehat{v^3_{N2}}(t,\xi)\,|\,\widehat{v}^3(t,\xi)\bigr)\bigr|
\leq \bigl\|\int_0^te^{-(t-\tau)|\xi_{\h}|^2}|\xi_{\h}||\xi|^{-1}\widehat{\vv v\otimes\vv v}(\tau,\xi)\,d\tau\bigr\|_{L^2_{\xi}}\|\widehat{\na_\h\vv v}(t)\|_{L^2}.
\eeq
Yet by applying H\"older's inequalities, we have
\beno\begin{aligned}
&\bigl\|\int_0^te^{- (t-\tau)|\xi_{\h}|^2}|\xi_{\h}||\xi|^{-1}\widehat{\vv v\otimes\vv v}(\tau,\xi)d\tau\bigr\|_{L^2_{\xi}}\\
&\leq \bigl\|\||\xi|^{-1}\|_{L^2_{\xi_3}} \bigl\|\int_0^te^{- (t-\tau)|\xi_{\h}|^2}|\xi_{\h}|\widehat{\vv v\otimes\vv v}(\tau,\xi)\,d\tau\bigr\|_{L^\infty_{\xi_3}}\bigr\|_{L^2_{\xi_{\h}}},
\end{aligned}\eeno
Observing that for $\delta>0$,
\beno\begin{aligned}
&\|e^{- (t-\tau)|\xi_{\h}|^2}|\xi_{\h}|^{2\delta}\|_{L^\infty_{\xi_{\h}}}\lesssim (t-\tau)^{-\delta},\\
&\||\xi|^{-1}\|_{L^2_{\xi_3}}=\Bigl(\int_{\R}\f{1}{|\xi_{\h}|^2+|\xi_3|^2}\,d\xi_3\Bigr)^{\f12}\lesssim|\xi_{\h}|^{-\f12},
\end{aligned}\eeno
so that we get, by using Young's inequality, that
\beq\label{A3}\begin{aligned}
\bigl\|\int_0^te^{- (t-\tau)|\xi_{\h}|^2}|\xi_{\h}||\xi|^{-1}\widehat{\vv v\otimes\vv v}(\tau,\xi)\bigr\|_{L^2_{\xi}}
\lesssim & \bigl\|\int_0^t(t-\tau)^{-\delta}\bigl|\xi_{\h}\bigr|^{\f12-2\delta}|\widehat{\vv v\otimes\vv v}(\tau,\xi)|\,d\tau\bigr\|_{L^2_{\xi_\h}(L^\infty_{\xi_3})}\\
\lesssim &\int_0^t(t-\tau)^{-\delta}\bigl\||D_{\h}|^{\f12-2\delta}(\vv v\otimes\vv v)(\tau,\cdot)\bigr\|_{L^1_\v(L^2_\h)}\,d\tau,
\end{aligned}\eeq
where $|D_h|^s$ is the Fourier multiplier with symbol $|\xi_\h|^s,$ that is 
\beno
|D_h|^sf\eqdefa\mathcal{F}^{-1}(|\xi_h|^s\hat{f}(\xi)),\quad\text{for } f\in\mathcal{S}'(\R^3).
\eeno
While it follows from the law of product in Sobolev spaces that for any fixed $\tau,x_3$ and for any $\d\in (0,3/4),$
\beno\begin{aligned}
\bigl\||D_{\h}|^{\f12-2\delta}(\vv v\otimes\vv v)(\tau,\cdot, x_3)\bigr\|_{L^2_\h}\lesssim \|\vv v(\tau,\cdot,x_3\|_{\dH^{\f34-\d}_\h}^2,
\end{aligned}\eeno
 so that
\beq\label{S4eq9}
\bigl\||D_{\h}|^{\f12-2\delta}(\vv v\otimes\vv v)(\tau,\cdot)\bigr\|_{L^1_\v(L^2_\h)}\lesssim \|\vv v(\tau,\cdot)\|_{\dH^{\f34-\d,0}}^2.
\eeq
By substituting \eqref{S4eq9} into \eqref{A3} and using \eqref{S4eq6} and \eqref{S4eq8}, we find
\beno\begin{aligned}
&\bigl\|\int_0^t e^{-(t-\tau)|\xi_{\h}|^2}|\xi_{\h}||\xi|^{-1}|\widehat{\vv v\otimes\vv v}(\tau,\xi)\,d\tau\bigr\|_{L^2_{\xi}}
\lesssim  \int_0^t(t-\tau)^{-\delta}\|\vv v(\tau)\|_{\dH^{\f34-\d,0}}^2\,d\tau\\
&\qquad\lesssim \int_0^t(t-\tau)^{-\delta}\|\vv v(\tau)\|_{L^2}^{\f12+2\d}\|\na_\h\vv v(\tau)\|_{L^2}^{\f32-2\d}\,d\tau
\leq CA_s\int_0^t(t-\tau)^{-\delta}\w{\tau}^{-s}\tau^{-\left(\f34-\d\right)}\,d\tau.
\end{aligned}\eeno
Notice that for $\delta\in(s-\f14,\f34)$, one has
\beno\begin{aligned}
\int_0^t(t-\tau)^{-\delta}\w{\tau}^{-s}\tau^{-\left(\f34-\d\right)}\,d\tau
\leq &\int_0^{\f{t}2}(t-\tau)^{-\delta}\w{\tau}^{-s}\tau^{-\left(\f34-\d\right)}\,d\tau\\
&+\int_{\f{t}2}^t(t-\tau)^{-\delta}\w{\tau}^{-s}\tau^{-\left(\f34-\d\right)}\,d\tau
\lesssim \w{t}^{-s}t^{\f14},
\end{aligned}\eeno
we find
\beno
\bigl\|\int_0^t e^{-(t-\tau)|\xi_{\h}|^2}|\xi_{\h}||\xi|^{-1}\widehat{\vv v\otimes\vv v}(\tau,\xi)\,d\tau\bigr\|_{L^2_{\xi}}
\leq CA_s\w{t}^{-s}t^{\f14},
\eeno
which together with \eqref{S4eq8} and \eqref{A2} ensures that
\beq\label{S4eq10}
\bigl|\bigl(\widehat{v^3_{N1}}(t,\xi)+\widehat{v^3_{N2}}(t,\xi)\,|\,\widehat{v}^3(t,\xi)\bigr)\bigr|
\leq CA_s^{\f32}\w{t}^{-\f32s}t^{-\f14}.
\eeq

By inserting the estimates \eqref{decay estimate for the linear part of v 3} and \eqref{S4eq10} into  \eqref{A1},  we achieve
\beq\label{A8}
\|v^3(t)\|_{L^2}^2\leq CB_s\w{t}^{-\f32s}t^{-\f14}\with B_s\eqdefa\|v_0^3\|_{L^2}^2+ \|\vv v_0\|_{\dH^{-s,-\f{s}{2}-\f14}}^2+A_s^{\f32},
\eeq
 for any $t\leq T^\star$.

\smallskip
\no
{\bf Step 5. The decay estimate of $\|\na_\h v^3(t)\|_{L^2}.$}
\smallskip

We first get, by a similar derivation of \eqref{A2}, that
\beq \label{S4eq11}
\f12\f{d}{dt}\|v^3(t)\|_{L^2}^2+\|\na_\h v^3\|_{L^2}^2\lesssim \bigl\||D_\h||D|^{-1}(\vv v\otimes\vv v)\bigr\|_{L^2}\|\na_\h\vv v\|_{L^2}.
\eeq
Yet along the same line to the proof of \eqref{S4eq10}, we find
\beno\begin{aligned}
&\bigl\||D_\h||D|^{-1}(\vv v\otimes\vv v)\bigr\|_{L^2}=\bigl\||\xi_\h||\xi|^{-1}\widehat{\vv v\otimes\vv v}\bigr\|_{L^2}
\leq \bigl\|\||\xi|^{-1}\|_{L^2_{\xi_3}}|\xi_\h|\|\widehat{\vv v\otimes\vv v}\|_{L^\infty_{\xi_3}}\bigr\|_{L^2_{\xi_\h}}\\
&\qquad\lesssim  \bigl\||\xi_\h|^{\f12}\widehat{\vv v\otimes\vv v}\bigr\|_{L^2_{\xi_\h}(L^\infty_{\xi_3})}
\lesssim \bigl\||D_\h|^{\f12}({\vv v\otimes\vv v})\bigr\|_{L^1_\v(L^2_\h)}
\lesssim \|\vv v\|_{\dH^{\f34,0}}^2\lesssim \|\vv v\|_{L^2}^{\f12}\|\na_\h\vv v\|_{L^2}^{\f32}.
\end{aligned}\eeno
By inserting the above estimate into \eqref{S4eq11} and integrating the resulting inequality over $[t/2,t],$ we achieve
\beno\begin{aligned}
\f12\|v^3(t)\|_{L^2}^2+\int_{\f{t}2}^t\|\na_\h v^3(t')\|_{L^2}^2\,dt'\leq &\f12\|v^3(t/2)\|_{L^2}^2
+C\int_{\f{t}2}^t\|\vv v(t')\|_{L^2}^{\f12}\|\na_\h\vv v(t')\|_{L^2}^{\f52}\,dt',
\end{aligned}\eeno
which together with \eqref{S4eq6}, \eqref{S4eq8} and \eqref{A8} implies
\beq \label{S4eq12}
\f12\|v^3(t)\|_{L^2}^2+\int_{\f{t}2}^t\|\na_\h v^3(t')\|_{L^2}^2\,dt'\leq B_s\w{t}^{-\f32s}t^{-\f14}\quad \forall\ t\leq T^\star.
\eeq

On the other hand, by taking $L^2$-inner product of the $v^3$ equation of \eqref{ANS} with $\D_\h v^3,$ we
deduce from a similar derivation of \eqref{A2} that
\beq
 \label{S4eq13}
\f12\f{d}{dt}\|\na_\h v^3(t)\|_{L^2}^2+\|\D_\h v^3\|_{L^2}^2\lesssim \bigl\||D_\h|^2|D|^{-1}(\vv v\otimes\vv v)\bigr\|_{L^2}\|\D_\h\vv v\|_{L^2}.
\eeq
We observe that
\beno\begin{aligned}
\bigl\||D_\h|^2|D|^{-1}(\vv v\otimes\vv v)\bigr\|_{L^2}=&\bigl\||\xi|^{-1}|\xi_\h|^2\widehat{\vv v\otimes\vv v}\bigr\|_{L^2}
\leq \bigl\|\||\xi|^{-1}\|_{L^2_{\xi_3}}|\xi_\h|^2\|\widehat{\vv v\otimes\vv v}\|_{L^\infty_{\xi_3}}\bigr\|_{L^2_{\xi_\h}}\\
\lesssim&  \bigl\||\xi_\h|^{\f32}\widehat{\vv v\otimes\vv v}\bigr\|_{L^2_{\xi_\h}(L^\infty_{\xi_3})}
\lesssim \bigl\||D_\h|^{\f12}({\vv v\otimes\na_\h \vv v})\bigr\|_{L^1_\v(L^2_\h)}\\
\lesssim &\|\vv v\|_{\dH^{\f34,0}}\|\na_\h\vv v\|_{\dH^{\f34,0}}\lesssim \|\vv v\|_{L^2}^{\f14}\|\na_\h\vv v\|_{L^2}\|\D_\h\vv v\|_{L^2}^{\f34}.
\end{aligned}\eeno
By inserting the above estimate into \eqref{S4eq13}, we find
\beno\begin{aligned}
\f12\f{d}{dt}\|\na_\h v^3(t)\|_{L^2}^2+\|\D_\h v^3\|_{L^2}^2\leq& C\|\vv v\|_{L^2}^{\f14}\|\na_\h\vv v\|_{L^2}\|\D_\h\vv v\|_{L^2}^{\f74}.
\end{aligned}\eeno
For any $0\leq t_0\leq t,$ by
multiplying the above inequality by $t-t_0$
and then integrating the resulting inequality over $[t_0,t],$ we obtain
\beno\begin{aligned}
&(t-t_0)\|\na_\h v^3(t)\|_{L^2}^2\leq \int_{t_0}^t\|\na_\h v^3(t')\|_{L^2}^2\,dt'
+C\int_{t_0}^t(t'-t_0)\|\vv v(t')\|_{L^2}^{\f14}\|\na_\h\vv v(t')\|_{L^2}\|\D_\h\vv v(t')\|_{L^2}^{\f74}\,dt'\\
&\qquad\lesssim \int_{t_0}^t\|\na_\h v^3(t')\|_{L^2}^2\,dt'+{\Bigl(\int_{t_0}^t(t'-t_0)\|\vv v(t')\|_{L^2}^2\|\na_\h\vv v(t')\|_{L^2}^8\,dt'\Bigr)^{\f18}
\Bigl(\int_{t_0}^t(t'-t_0)\|\D_\h\vv v(t')\|_{L^2}^2\,dt'\Bigr)^{\f78}.}
\end{aligned}\eeno
Taking $t_0=\f{t}2$  in the above inequality and using \eqref{S4eq6}, \eqref{S4eq8} and \eqref{S4eq12}, we find
\beq \label{S4eq14}
t\|\na_\h v^3(t)\|_{L^2}^2\leq CB_s\w{t}^{-\f32s}t^{-\f14} \quad \forall\ t\leq T^\star.
\eeq
 This together with \eqref{A8} ensures that  \eqref{S1eq4} holds for $t\leq T^\star.$

\smallskip
\no{\bf Step 6. Closing of the continuity argument.}
\smallskip

To close the continuity argument, we shall use \eqref{Estimate for partial 3 v 1} in  Lemma \ref{lemma for partial 3 v}.
We first observe that for any $s_1>2,$
\beno\begin{aligned}
\|\p_3^2\vv v\|_{L^2}\leq &\|\p_3\vv v\|_{L^2}^{\f{s_1-2}{s_1-1}}\|\vv v\|_{\dH^{0,s_1}}^{\f{1}{s_1-1}}
\leq\|\p_3\vv v\|_{\dH^{-\f12,0}}^{\f{s_1-2}{2(s_1-1)}}\|\p_3\vv v\|_{\dH^{\f12,0}}^{\f{s_1-2}{2(s_1-1)}}\|\vv v\|_{\dH^{0,s_1}}^{\f{1}{s_1-1}},
\end{aligned}\eeno
so that we get, by applying Young's inequality, that
\beno\begin{aligned}
\|\na_{\h}& v^3\|_{L^2}^{\f12}\|\dive_{\h}\vv v^\h\|_{L^2}^{\f12}\|\p_3^2\vv v\|_{L^2}\|\p_3\vv v\|_{\dH^{-\f12,0}}\\
\leq &\|\na_{\h} v^3\|_{L^2}^{\f12}\|\dive_{\h}\vv v^\h\|_{L^2}^{\f12}\|\p_3\vv v\|_{\dH^{-\f12,0}}^{\f{3s_1-4}{2(s_1-1)}}\|\vv v\|_{\dH^{0,s_1}}^{\f{1}{s_1-1}}\|\p_3\vv v\|_{\dH^{\f12,0}}^{\f{s_1-2}{2(s_1-1)}}\\
\leq &\f12 \|\p_3\vv v\|_{\dH^{\f12,0}}^2+C\bigl(\|\na_{\h} v^3\|_{L^2}\|\dive_{\h}\vv v^\h\|_{L^2}\bigr)^{\f{2(s_1-1)}{3s_1-2}}\bigl(\|\p_3\vv v\|_{\dH^{-\f12,0}}^{2}+\|\vv v\|_{\dH^{0,s_1}}^{2}\bigr).
\end{aligned}\eeno
Substituting the above inequality into \eqref{Estimate for partial 3 v 1} gives rise to
\beno\begin{aligned}
\f{d}{dt}\|\p_3\vv v\|_{\dot{H}^{-\f12,0}}^2+\|\p_3\vv v\|_{\dot{H}^{\f12,0}}^2
&\leq C\|\na_{\h}\vv v\|_{\cB^{0,\f12}}^2\|\p_3\vv v\|_{\dot{H}^{-\f12,0}}^2\\
&\quad
+C\bigl(\|\na_{\h} v^3\|_{L^2}\|\dive_{\h}\vv v^\h\|_{L^2}\bigr)^{\f{2(s_1-1)}{3s_1-2}}\bigl(\|\p_3\vv v\|_{\dH^{-\f12,0}}^{2}+\|\vv v\|_{\dH^{0,s_1}}^{2}\bigr).
\end{aligned}\eeno
Applying Gronwall's inequality leads to
\beq\label{S4eq15}
\begin{split}
\|\p_3\vv v&\|_{L^\infty_t(\dot{H}^{-\f12,0})}^2+\|\p_3\vv v\|_{L^2_t(\dot{H}^{\f12,0})}^2\\
\leq & C\Bigl(\|\p_3\vv v_0\|_{\dot{H}^{-\f12,0}}^2+\int_0^t\bigl(\|\na_{\h}\vv v^3\|_{L^2}\|\dive_{\h}\vv v^\h\|_{L^2}\bigr)^{\f{2(s_1-1)}{3s_1-2}}\,dt'\|\vv v\|_{L^\infty_t(\dH^{0,s_1})}^{2}\Bigr)\\
&\qquad\qquad\times\exp\Bigl(C\|\na_{\h}\vv v\|_{L^2_t(\cB^{0,\f12})}^2+C\int_0^t\bigl(\|\na_{\h} v^3\|_{L^2}\|\dive_{\h}\vv v^\h\|_{L^2}\bigr)^{\f{2(s_1-1)}{3s_1-2}}\,dt'\Bigr).
\end{split}
\eeq
Yet for $t\leq T^\star,$ it follows from \eqref{S4eq8} and \eqref{S4eq14} that
\beno\begin{aligned}
\int_0^t\bigl(\|\na_{\h} v^3\|_{L^2}\|\dive_{\h}\vv v^\h\|_{L^2}\bigr)^{\f{2(s_1-1)}{3s_1-2}}\,dt'
\leq & C\bigl(A_sB_s\bigr)^{\f{s_1-1}{3s_1-2}}\int_0^t\Bigl(\w{\tau}^{-\f54s}\tau^{-\f98}\Bigr)^{\f{2(s_1-1)}{3s_1-2}}\,d\tau
\leq C\bigl(A_s B_s\bigr)^{\f{s_1-1}{3s_1-2}},
\end{aligned}\eeno
if
\beno\begin{aligned}
\left(\f98+\f54s\right){\f{2(s_1-1)}{3s_1-2}}>1\Leftrightarrow s>\f{1+3s_1}{10(s_1-1)}.
\end{aligned}\eeno
Therefore thanks to \eqref{S4eq2a}, we deduce from \eqref{S4eq15} that for $t\leq T^\star$
\beq \label{S4eq20}
\begin{split}
\|\p_3\vv v&\|_{L^\infty_t(\dot{H}^{-\f12,0})}^2+\|\p_3\vv v\|_{L^2_t(\dot{H}^{\f12,0})}^2\\
\leq & C\Bigl(\|\p_3\vv v_0\|_{\dot{H}^{-\f12,0}}^2+\|\vv v_0\|_{\dH^{0,s_1}}^{2}\bigl(A_s B_s\bigr)^{\f{s_1-1}{3s_1-2}}\Bigr)
\exp\Bigl(Cc_0+C\bigl(A_s B_s\bigr)^{\f{s_1-1}{3s_1-2}}\Bigr).
\end{split} \eeq
In particular, under the assumption \eqref{initial ansatz a}, we infer
\beq\label{S4eq16}
\begin{split}
\|\p_3\vv v&\|_{L^\infty_t(\dot{H}^{-\f12,0})}^2+\|\p_3\vv v\|_{L^2_t(\dot{H}^{\f12,0})}^2
\leq \f\e2\quad\mbox{for} \ t\leq T^\star,
\end{split}
\eeq
which contradicts with \eqref{S4eq4}. This in turn shows that $T^\star=\infty.$ Furthermore, \eqref{S4eq2a}, \eqref{S4eq3} and \eqref{S4eq20}
ensures \eqref{S1eq2}, and there hold \eqref{S1eq3a} and \eqref{S1eq4}. To complete the proof of Theorem \ref{decay theorem},
it remains to prove \eqref{S1eq3b}.

\smallskip

\no{\bf Step 7. The decay estimate of $\|\p_3\vv v(t)\|_{L^2}$.}
\smallskip

We first deduce from \eqref{S3eq1} and \eqref{S3eq10a} that
\beq\label{S4eq17}
\f{d}{dt}\|\p_3\vv v(t)\|_{L^2}^2+ \|\na_\h \p_3\vv v\|_{L^2}^2\leq {C}\|\na_\h \vv v\|_{\cB^{0,\f12}}^2\|\p_3\vv v\|_{L^2}^2.
 \eeq

Let us denote
\beno\begin{aligned}
&X(t)\eqdefa e^{-C\int_0^t\|\na_{\h}\vv v(\tau)\|_{\cB^{0,\f12}}^2\,d\tau}\|\p_3\vv v(t)\|_{L^2}^2\andf D(t)\eqdefa e^{-C\int_0^t\|\na_{\h}\vv v\|_{\cB^{0,\f12}}^2\,d\tau}\|\na_{\h}\p_3\vv v(t)\|_{L^2}^2.
\end{aligned}\eeno
Then we deduce from \eqref{S4eq17} that
\beq\label{s12}
\f{d}{dt}X(t)+ D(t)\leq 0.
\eeq
It follows from \eqref{S3eq1} that
\beno
\int_0^t\|\na_{\h}\vv v(\tau)\|_{\cB^{0,\f12}}^2\,d\tau\leq C\|\vv v_0\|_{\cB^{0,\f12}}^2\leq Cc_0^2,
\eeno
so that there holds
\beno
e^{-C\int_0^t\|\na_{\h}\vv v(\tau)\|_{\cB^{0,\f12}}^2\,d\tau}\geq e^{-Cc_0^2}\geq e^{-1} \quad\mbox{if}\ \ Cc_0^2\leq 1,
\eeno
which implies
\beq\label{s13}
e^{-1}\|\p_3\vv v(t)\|_{L^2}^2\leq X(t)\leq\|\p_3\vv v(t)\|_{L^2}^2,\andf
e^{-1}\|\na_{\h}\p_3\vv v(t)\|_{L^2}^2\leq D(t)\leq\|\na_{\h}\p_3\vv v(t)\|_{L^2}^2.
\eeq

On the other hand, we observe that
\beno
\|\p_3\vv v(t)\|_{L^2}\leq\|\p_3\vv v(t)\|_{\dot{H}^{-\f12,0}}^{\f23}\|\na_{\h}\p_3\vv v(t)\|_{L^2}^{\f13},
\eeno
from which and \eqref{S1eq2c}, \eqref{s13}, we infer
\beno
D(t)\geq e^{-1}\|\na_{\h}\p_3\vv v(t)\|_{L^2}^2\geq e^{-1}\f{\|\p_3\vv v(t)\|_{L^2}^6}{\|\p_3\vv v(t)\|_{\dot{H}^{-\f12,0}}^4}
\geq e^{-1}\f{X^3(t)}{\|\p_3\vv v(t)\|_{\dot{H}^{-\f12,0}}^4}\geq (e\cE_0^2)^{-1}X^3(t).
\eeno
Then we deduce from \eqref{s12} that
\beno
\f{d}{dt}X(t)+(e\cE_0^2)^{-1}X^3(t)\leq0,
\eeno
which implies
\beno
X(t)\leq\Bigl(\f{1}{X(0)^{-2}+(e\cE_0^2)^{-1} t}\Bigr)^{\f12}=
\Bigl(\f{1}{\|\p_3\vv v_0\|_{L^2}^{-4}+(e\cE_0^2)^{-1} t}\Bigr)^{\f12}\leq\Bigl(\|\p_3\vv v_0\|_{L^2}^2+\sqrt{e\cE_0^2}
\Bigr)\w{t}^{-\f12},
\eeno
This together with  \eqref{s13} ensures \eqref{S1eq3b}.

\medskip

As a consequence, we complete the proof of Theorem \ref{decay theorem}.
\end{proof}

\appendix

\setcounter{equation}{0}
\section{Tool box on Littlewood-Paley theory}

For the convenience of readers, we collect some basic facts on anisotropic
Littlewood-Paley theory in this section.
We first observe from Definition \ref{defanisob} and \eqref{defparaproduct} that for any $s_1, s_2\in\R,$
\beno
\|f\|_{\dH^{s_1,s_2}}^2\sim\sum_{k,\ell\in\Z}2^{2ks_1}2^{2\ell s_2}\|\D_k^\h\D_\ell^\v f\|_{L^2}^2.
\eeno

 We  recall the
following anisotropic version of Bernstein type lemma from \cite{BCD, CZ1, Paicu}:

\begin{lemma}\label{L2} {\sl Let $\cB_{\h}$ (resp.~$\cB_{\v}$) be a ball
of~$\R^2$ (resp.~$\R$), and~$\cC_{\h}$ (resp.~$\cC_{\v}$) a ring
of~$\R^2$ (resp.~$\R$); let~$1\leq p_2\leq p_1\leq \infty$ and
~$1\leq q_2\leq q_1\leq \infty.$ Then there hold:
\smallbreak\noindent If the support of~$\wh a$ is included
in~$2^k\cB_{\h}$, then
\[
\|\partial_{\h}^\alpha a\|_{L^{p_1}_{\h}(L^{q_1}_{\v})} \lesssim
2^{k\left(|\al|+2\left(\frac1{p_2}-\frac1{p_1}\right)\right)}
\|a\|_{L^{p_2}_{\h}(L^{q_1}_{\v})},\quad\text{for}\quad \p_{\h}=(\p_1,\p_2).
\]
If the support of~$\wh a$ is included in~$2^\ell\cB_{\v}$, then
\[
\|\partial_3^\beta a\|_{L^{p_1}_{\h}(L^{q_1}_{\v})} \lesssim
2^{\ell\left(\beta+\left(\frac1{q_2}-\frac1{q_1}\right)\right)} \|
a\|_{L^{p_1}_{\h}(L^{q_2}_{\v})}.
\]
If the support of~$\wh a$ is included in~$2^k\cC_{\h}$, then
\[
\|a\|_{L^{p_1}_{\h}(L^{q_1}_{\v})} \lesssim 2^{-kN} \|\partial_{\h}^N
a\|_{L^{p_1}_{\h}(L^{q_1}_{\v})}.
\]
If the support of~$\wh a$ is included in~$2^\ell\cC_{\v}$, then
\[
\|a\|_{L^{p_1}_{\h}(L^{q_1}_{\v})} \lesssim 2^{-\ell N} \|\partial_3^N
a\|_{L^{p_1}_{\h}(L^{q_1}_{\v})}.
\]}
\end{lemma}

To deal with the estimate of product of two distributions, we constantly
use the following para-differential decomposition from \cite{Bo81} in the horizontal variables: for any functions $f,g\in\mathcal{S}'(\R^3)$,
\beq \label{bony}
fg=T^\h_fg+{T}^\h_gf+R^\h(f,g),
\eeq
where
\beno\begin{aligned}
&T^\h_fg\eqdefa\sum_{k\in\Z}S^\h_{k-1}f\D^\h_k g,\quad R^\h(f,g)\eqdefa\sum_{k\in\Z}\Delta^\h_k f\tilde{\Delta}^\h_k g,\quad
\hbox{with} \quad
\wt{\Delta}^\h_k g\eqdefa\sum_{k'=k-1}^{k+1}\D^\h_{k'} g.
\end{aligned}\eeno
We also employ Bony's decomposition in the vertical variable.

\smallskip

The following technical lemmas are very useful in this context.
The first one is concerned with the commutator's estimates involving $\D^\h_k$ and $\D^v_\ell$.
Indeed it follows from
 the classic commutator's estimate from \cite{BCD} and H\"older's inequality that

\begin{lemma}\label{commutator lemma}
{\sl Let $p,q,r,s\in[1,\infty]$ which satisfy $\f{1}{p}=\f{1}{r}+\f{1}{s}$.
Then for any $f,g\in\mathcal{S}(\R^3)$, $j,k\in\Z$, there hold
\beq\label{commutator estimate}\begin{aligned}
&\|[\Delta_\ell^\v; f]g\|_{L^p_{\h}(L^q_{\v})}\lesssim 2^{-\ell}\|\na_{\h}f\|_{L^r_{\h}(L^\infty_{\v})}\|g\|_{L^s_{\h}(L^q_{\v})},\\
&\|[\Delta_k^h; f]g\|_{L^p_{\v}(L^q_{\h})}\lesssim 2^{-k}\|\na_{\h}f\|_{L^r_{\v}(L^\infty_{\h})}\|g\|_{L^s_{\v}(L^q_{\h})}.
\end{aligned}\eeq}
\end{lemma}

By virtue of Lemma \ref{L2} and the following 2D interpolation inequality
\beno
\|a\|_{L^4(\R^2)}\lesssim \|a\|_{L^2(\R^2)}^{\f12}\|\na_\h a\|_{L^2(\R^2)}^{\f12},
\eeno
we deduce that

\begin{lemma}\label{interpolation lemma}
{\sl For any $u\in\mathcal{S}(\R^3)$, there hold
\beq\label{interpolation inequality}\begin{aligned}
&\|u\|_{L^4_{\h}(L^\infty_{\v})}\lesssim\|u\|_{L^4_{\h}(\cB^{0,\f12})_{\v}}\lesssim
\|u\|_{\cB^{0,\f12}_{2,1}}^{\f12}\|\na_h u\|_{\cB^{0,\f12}_{2,1}}^{\f12},\\
&\|u\|_{L^4_{\h}(L^2_{\v})}\lesssim\|u\|_{L^2_{\v}(L^4_{\h})}\lesssim\|u\|_{L^2_{\v}(\dot{H}^{\f12}_{\h})}\lesssim\|u\|_{L^2}^{\f12}\|\na_{\h}u\|_{L^2}^{\f12},\\
&\|u\|_{L^\infty_{\v}(L^2_{\h})}\lesssim\|u\|_{L^2_{\h}(L^\infty_{\v})}
\lesssim\|u\|_{\cB^{0,\f12}_{2,1}}\lesssim\|u\|_{L^2}^{\f12}\|\p_3u\|_{L^2}^{\f12}.
\end{aligned}\eeq}
\end{lemma}

\vspace{0.5cm}
\noindent {\bf Acknowledgments.} Li Xu is supported by NSF of China under grant 11671383.  Ping Zhang is partially supported by K. C. Wong Education Foundation and NSF of China under Grants   11731007, 12031006 and 11688101.

\end{document}